\documentclass[11pt]{amsart}
\usepackage{amsmath}
\usepackage{amssymb}
\usepackage{tabularx}
\usepackage{enumerate}
\usepackage[dvipdfm]{graphicx}
\usepackage{texdraw}
\usepackage{color}

\usepackage{mathrsfs}

\usepackage{amsfonts,amssymb,amsmath}
\usepackage{epsfig}

\makeatletter
\@addtoreset{equation}{section}

\makeatother
\newtheorem{thm}{Theorem}[section]
\newtheorem{lem}[thm]{Lemma}

\newtheorem{prop}[thm]{Proposition}
\newtheorem{remark}[thm]{Remark}

\newcommand{\R}{\mathbb{R}}

\newcommand{\wt}{\widetilde}
\newcommand{\wh}{\widehat}
\newcommand{\ov}{\overline}
\newcommand{\HH}{(H_0)}
\newcommand{\HHH}{(H_1)}

\newcommand{\PP}{\mathcal{P}}

\newcommand{\BBB}{\mathscr{B}}

\begin{document}
\title[A Fisher-KPP equation with free boundaries and time-periodic advections]
{A diffusive Fisher-KPP equation with free boundaries and time-periodic advections$^\S$}
 \thanks{$\S$ This research was partly supported by  NSFC (No. 11271285). }
\author[N. Sun, B. Lou, M. Zhou]{Ningkui Sun$^\dag$, Bendong Lou$^{\ddag, *}$ and Maolin Zhou$^\sharp$}
\thanks{$\dag$ School of Mathematical Science, Shandong Normal University, Jinan 250014, China.}
\thanks{$\dag$ Mathematics \& Science College, Shanghai Normal University, Shanghai 200234, China.}
\thanks{$\sharp$ School of Science and Technology, University of New England, Armidale, NSW 2351, Australia.}
\thanks{{\bf Emails:} {\sf sunnk1987@163.com} (N. Sun), {\sf bendonglou@sina.com} (B. Lou), {\sf mzhou6@une.edu.au} (M. Zhou)}
\thanks{$*$ Corresponding author.}
\date{}

\begin{abstract}
We consider a reaction-diffusion-advection equation of the form: $u_t=u_{xx}-\beta(t)u_x+f(t,u)$ for $x\in (g(t),h(t))$,
where $\beta(t)$ is a $T$-periodic function representing the intensity of the advection, $f(t,u)$ is a Fisher-KPP type
of nonlinearity, $T$-periodic in $t$, $g(t)$ and $h(t)$ are two free boundaries satisfying Stefan conditions. This equation can be used to
describe the population dynamics in time-periodic environment with advection. Its homogeneous version
(that is, both $\beta$ and $f$ are independent of $t$) was recently studied by Gu, Lou and Zhou \cite{GLZ}.
In this paper we consider the time-periodic case and study the long time behavior of the solutions. We show that
a vanishing-spreading dichotomy result holds when $\beta$ is small;
a vanishing-transition-virtual spreading trichotomy result holds when $\beta$ is a medium-sized function; all solutions vanish
when $\beta$ is large. Here the partition of $\beta(t)$ is much more complicated than the case when $\beta$ is a
real number, since it depends not only on the \lq\lq size" $\bar{\beta}:= \frac{1}{T}\int_0^T \beta(t) dt$
of $\beta(t)$ but also on its \lq\lq shape" $\tilde{\beta}(t) := \beta(t) - \bar{\beta}$.
\end{abstract}

\subjclass[2010]{35K20, 35K55, 35R35, 35B40}
\keywords{reaction-diffusion-advection equation, Fisher-KPP equation, free boundary problem,
time-periodic environment, long time behavior, spreading phenomena}
\maketitle


\section{Introduction}\label{sec:intr}

The study of spreading processes by using reaction diffusion equations traces back to the pioneering works of
Fisher \cite{Fisher}, and Kolmogorov, Petrovski and Piskunov \cite{KPP}. They introduced the equation $u_t = u_{xx} + u(1-u)$
to model the spread of advantageous genetic trait in a population, and studied traveling wave solutions of the form
$u(t,x)=\phi(x-ct)$. In 1970s', Aronson and Weinberger \cite{AW1, AW2} gave a systematic investigation on the Cauchy problem
of $u_t= u_{xx} +f(u)$. In particular, when $f$ is a monostable type of nonlinearity like $u(1-u)$, they proved the so-called
{\it hair-trigger effect}, which says that {\it spreading} always happens (i.e. $\lim_{t\to \infty} u(t,x)=1$) for the solution
starting from any nonnegative and compactly supported initial data (no matter how small it is). Furthermore, they showed that
the traveling wave with minimal speed can be used to characterize the spreading of a species.

In this paper we consider the population dynamics in time-periodic advective environments, which means that the spreading
of a species is affected by a time-periodic advection. In the field of ecology, organisms can often sense and respond to
local environmental cues by moving towards favorable habitats, and these movement usually depend
upon a combination of local biotic and abiotic factors such as stream, climate, food and predators.
For example, some diseases spread along the wind direction. More examples can be found in \cite{GLZ} and references therein.
From a mathematical point of view, to involve the influence of advection, one of the simplest but probably still realistic
approaches is to assume that species can move up along the gradient of the density. The equation $u_t = u_{xx} -\beta(t) u_x +f(t,u)$
is such an example. Note that, in a moving coordinate frame $y :=x - \int_0^t \beta (s)ds$, this equation reduces to one without
advection: $w_t = w_{yy} + f(t,w)$ for $w(t,y)= u(t,x)$. Hence, for the Cauchy problem, there is nothing new in mathematics
to be studied. This paper considers the equation in a variable domain with free boundaries. In most spreading processes in
the natural world, a spreading front can be observed. In one space dimension case, if the species initially occupies an interval
$(-h_0, h_0)$, as time $t$ increases from 0, it is natural to expect the end points of the habitat evolve into two spreading fronts:
$x=g(t)$ on the left and $x=h(t)$ on the right. To determine how these fronts
evolve with time, we assume that they invade at a speed that is proportional to the spatial gradient of the density function $u$
there, which gives the following free boundary problem
\begin{equation}\label{p}
\left\{
\begin{array}{ll}
 u_t =u_{xx}- \beta (t) u_{x}+f(t,u), &  g(t)< x<h(t),\; t>0,\\
 u(t,g(t))=0,\ \ g'(t)=-\mu (t) u_x(t, g(t)), & t>0,\\
 u(t,h(t))=0,\ \ h'(t)=-\mu (t) u_x  (t, h(t)) , & t>0,\\
-g(0)=h(0)= h_0,\ \ u(0,x) =u_0 (x),& -h_0\leqslant x \leqslant h_0,
\end{array}
\right.
\tag{$P$}
\end{equation}
where $\mu,\, \beta$ and $f$ are $T$-periodic functions in time $t$, $h_0>0$ and $u_0$ is a nonnegative function with support in
$[-h_0,h_0]$. We remark that this problem can be deduced as a spatial segregation limit of competition systems
(cf. \cite{HHP}). Namely, when we take the singular limit as the competition parameter goes to infinity in certain
competition systems, free boundaries will appear which separate a competitor from another and evolve
according to a law like that in \eqref{p}. On the other hand, the (Stefan) free boundary conditions in \eqref{p}
can also be derived from the Fick's diffusion law (cf. \cite{BDK}).

When $\beta \equiv 0$, that is, there is no advection in the environment, Du and Lin \cite{DuLin} studied
the problem \eqref{p} for the case where $\mu(t)=const.$ and $f(t,u)= u(a-bu)$ ($a,\, b>0$ are constants).
They proved that, if $h_0$ is small, then spreading happens (i.e., $h(t), -g(t)\to \infty$ and $u(t,\cdot)\to a/b$
as $t\to \infty$ locally uniformly in $\R$) when $\mu$ is large, and vanishing happens (i.e., $h(t)-g(t)$ is bounded
and $u(t,\cdot)\to 0$ uniformly as $t\to \infty$) when $\mu$ is small. The vanishing phenomena is a remarkable
result since it shows that the presence of free boundaries makes spreading difficult and the
{\it hair-trigger effect} can be avoided for some initial data with narrow support.
Later, some authors considered the problem in time dependent environments. Among them,
Du, Guo and Peng \cite{DGP} considered the time-periodic problem, and Li, Liang and Shen
\cite{LLS1} considered the almost time-periodic problem, both for the environments without advection
(i.e., $\beta\equiv 0$).

When $\beta\not\equiv 0$, that is, there is an advection in the environment, some special cases
of \eqref{p} were studied by some authors.
Among them, Gu, Lou and Zhou \cite{GLZ} gave a rather complete description
for the long time behavior of the solutions to the homogeneous version of \eqref{p}.
When $f$ is a Fisher-KPP type of nonlinearity like $f(u)=u(1-u)$, they found that, besides the
minimal speed $c_0 := 2\sqrt{f'(0)}$ of traveling waves, there is another
important parameter $\beta^*>c_0$ which
affects the dynamics of the solutions significantly. More precisely,
\begin{itemize}
\item[(i)] in small advection case $\beta\in [0,c_0)$, there is a dichotomy result
(cf. \cite[Theorem 2.1]{GLZ}): either \emph{vanishing} or \emph{spreading}
happens for a solution;

\item[(ii)]  in medium-sized advection case $\beta \in [c_0, \beta^*)$,
there is a trichotomy result (cf. \cite[Theorem 2.2]{GLZ}): either \emph{vanishing} happens,
or \emph{virtual spreading} happens (which means that, as $t\to \infty$, $g(t)\to g_\infty >-\infty,\ h(t)\to \infty$,
$u(t,\cdot)\to0$ locally uniformly in $(g_\infty, \infty)$ and $u(t,\cdot+c_1 t)\to1$ locally uniformly
in $\R$ for some $c_1 >\beta -c_0$), or the solution is a \emph{transition} one in the sense that,
when $\beta \in (c_0, \beta^*)$, $u(t, \cdot+o(t))$ converges to a {\it tadpole-like traveling semi-wave}
$\mathcal{V}(\cdot-(\beta-c_0)t)$, and that, when $\beta = c_0$, $u(t,\cdot)\to 0$ uniformly
in $[g(t),h(t)]$ with $h(t)\to \infty$;

\item[(iii)] in large advection case $\beta \geqslant \beta^*$,
\emph{vanishing} happens for all the solutions  (cf. \cite[Theorem 2.4]{GLZ}).
\end{itemize}

In this paper we study the problem \eqref{p} with time-periodic coefficients. Throughout this paper, we
assume that $T>0$ is a given constant and use the following notation:
$$
\begin{array}{l}
\PP:= \{p(t)\in C^{\nu/2} ([0,T])  : p(0) =p(T)\} \mbox{ for some } \nu\in (0,1);\\
\mbox{for each } p\in \PP,\ \ \mbox{denote } \bar{p}:= \frac{1}{T} \int_0^T p(t) dt \mbox{ and } \tilde{p}(t) := p(t) -\bar{p};\\
\PP^0 := \{ p\in \PP  : \bar{p} =0\},\quad \PP_+ := \{p\in \PP  : p(t)>0 \mbox{ for all } t\in [0,T] \}.
\end{array}
$$
Our basic assumptions is the following
$$
(H_0)\  \hskip 16mm
\left\{
 \begin{array}{l}
  \beta\in \PP \mbox{ with } \bar{\beta} \geqslant 0; \ \ \mu \in \PP_+ ;\ f(t,0)\equiv 0;\\
  f(t,u)\in C^{\nu/2, 1+\nu/2}([0,T]\times \R) \mbox{ for some } \nu \in (0,1),\ T\mbox{-periodic in } t;\\
  a(t):= f_u (t,0)\in \PP_+; \mbox{ and for any } t\in [0,T],  f(t,u)<0 \mbox{ for } u>1,\\
     f(t,u)/u \mbox{ is strictly decreasing in } u >0.
  \end{array}
 \right. \ \hskip 15mm \hfill
$$
This condition implies that $f(t,u)$ is positive for small $u$, negative for $u>1$ and $f(t,u)\leqslant  a(t) u$.
Hence $f(t,u)$ is a Fisher-KPP type of nonlinearity. Typical example of such $f$ is $f=u(a(t) -b(t)u)$ for some
$a,\, b\in \PP_+$. Note that the assumption $\bar{\beta}\geqslant 0$ is not an essential one,
since this assumption is used only to indicate that the advection makes the rightward motion
easier than the leftward one.
In the converse case: $\bar{\beta}<0$, all the conclusions in this paper remain valid as long as the right
and the left directions are exchanged.

We now sketch the influence of the advection intensity $\beta$ on the spreading of the species. For this purpose
we need three special solutions (see details in Section \ref{prelis}). (1) The unique positive {\it $T$-periodic solution}
$P(t)$ of the ODE $u_t = f(t,u)$. (2) The {\it periodic traveling wave}
$Q\big(t,x+\bar{c}t-\int_0^t \beta(s)ds \big)$ of \eqref{p}$_1$ (hereafter, we use \eqref{p}$_1$ to denote
the equation in \eqref{p}), where $\bar{c}:=2 \sqrt{\bar{a}}$ and $Q(t,z)$ denotes the unique solution of
\begin{equation}\label{p of Q}
\left\{
 \begin{array}{l}
v_t = v_{zz} - \bar{c} v_z +f(t,v) \mbox{ for } t, z\in \R, \\
v(t,-\infty)=0,\  v(t,\infty)= P(t) \mbox{ and }
v (0,0)= \frac12 \min_{t\in [0,T]}P(t).
\end{array}
\right.
\end{equation}
(3) {\it Periodic rightward traveling semi-wave} $U \big(t,R(t)- x\big)$, which is defined by the solution to the problem
\begin{equation}\label{p-sw-k}
\left\{
\begin{array}{ll}
U_t = U_{zz} + [\beta -r] U_z + f(t,U), &  t\in [0,T],\ z>0,\\
U(t,0)=0,\ \ U(t, \infty)= P(t),  & t\in [0,T],\\
U(0,z) = U(T,z),\ U_z(t,z)>0, & t\in [0,T],\ z \geqslant 0,\\
r(t) = \mu(t) U_z(t,0), & t\in [0,T].
\end{array}
\right.
\end{equation}
In Section \ref{prelis}, we will see that when $\beta\in\mathcal{P}$ satisfies $\bar{\beta} \geqslant 0$, the problem \eqref{p-sw-k} has a unique solution pair $(r,\, U)$ with $r =r(t;\beta)\in \PP_+$.
Then, with $R(t):=\int_0^t r(s;\beta)ds$, the function $u(t,x)=U (t,R(t)-x)$ satisfies \eqref{p}$_1$, $u(t,R(t))=0$ and $R'(t)=-\mu(t)u_x(t,R(t))$. As in \cite{DuLou}, we call $u=U(t,R(t)-x)$ a periodic rightward traveling semi-wave since it is
defined only for $x\leqslant  R(t)$ and $U(t,z)$ is periodic in $t$.

The long time behavior of the solutions of \eqref{p} is quite different when $\beta$ is a small,
or a medium-sized, or a large function. The partition for $\beta(t)$, however,
is much more complicated than the homogeneous case (cf. \cite{GLZ}) since
not only the ``size" $\bar{\beta}$ but also the ``shape" $\tilde{\beta}$ is involved.
According to our study we find that $\tilde{\beta}$ and $\bar{\beta}$ should be considered separately.
In fact, for each given ``shape" $\theta\in \PP^0$, if we consider only $\beta$ with the ``shape" $\theta$
(namely, consider $\beta = b+\theta$ for $b\in [0,\infty)$, hence $\bar{\beta} =b$ and $\tilde{\beta} =\theta$),
then we will show in Section 3 that there exist two critical values
for $b$. The first one is $\bar{c} := 2\sqrt{\bar{a}}$ (independent of the ``shape" $\theta$),
and $\bar{c} +\theta$ is the function partitioning small $\beta$ and medium-sized $\beta$ in the set
$\{\beta = b+\theta : b\geqslant 0\}$.
The second critical value for $b$ is $B(\theta)$, which depends on $\theta$ and is bigger than $\bar{c}$, and
$\beta^* := B(\theta) +\theta$ is the function partitioning medium-sized $\beta$ and large $\beta$ in the set
$\{\beta = b+\theta : b\geqslant 0\}$.
Here $B(\theta)$ is the unique zero of the increasing function $y(b):= b -\bar{c} -\overline{r(t;b+\theta)}$
in $[0,\infty)$. Therefore, with $\bar{r} =\ov{r(t;\beta)}$, we have
$$
\bar{\beta} -\bar{c} < \bar{r} \mbox{ when } \bar{\beta} < B(\tilde{\beta}),\quad
\bar{\beta} -\bar{c} = \bar{r} \mbox{ when } \bar{\beta} = B(\tilde{\beta}) \mbox{ and }
\bar{\beta} -\bar{c} > \bar{r} \mbox{ when } \bar{\beta} > B(\tilde{\beta})
$$
(see Lemmas \ref{lem:B(theta)} and \ref{lem:property of beta*} for details.)

The long time behavior of the solutions to \eqref{p} depends on the signs of
$\bar{\beta}-\bar{c}$ and $\bar{\beta}- B(\tilde{\beta})$.

{\it Case 1: $\bar{\beta}\in[0,\bar{c})$}. In this case the periodic traveling wave $Q\big(t,x+\bar{c}t-\int_0^t\beta(s)ds\big)$ moves leftward. This indicates that spreading of the solution on
the left side is possible for solutions starting from large initial data. Therefore {\it spreading} (i.e., $u\to P(t)$ as $t\to\infty$)
may happen in this case (see Theorem \ref{thm:small beta} below).

{\it Case 2: $\bar{\beta} \geqslant  \bar{c}$}. To explain the influence
of $\beta$ intuitively, we consider a solution of \eqref{p} with
a {\it front} (i.e., a sharp decreasing part) on the right side
and a {\it back} (i.e., a sharp increasing part) on the left side. As can be expected,
when $t\gg 1$ the front $\approx U(t, R(t)-x)$ and it moves rightward with speed $\approx \overline{r(t;\beta)}$,
the back $\approx Q\big(t,x+\bar{c}t -\int_0^t \beta(s) ds \big)$ and it also moves rightward with
speed $\approx \bar{\beta} -\bar{c}$.
The latter indicates that $u\to 0$ in $L^\infty_{loc}$ topology.

{\it Subcase 2.1: $\bar{c} \leqslant \bar{\beta} < B(\tilde{\beta})$}. In this case we have
$\bar{\beta}-\bar{c }< \overline{r(t;\beta)}$, that is, the front moves rightward faster than the back.
Hence the solution have enough space between the back and the front to grow up.
In fact we will show that, in $L^\infty_{loc}$ topology, $u(t,x-c_1 t)\to P(t)$ as $t\to \infty$ for
some $c_1>\bar{\beta}-\bar{c}$, though $u(t,x)\to0$ locally uniformly.
We will call this phenomenon as {\it virtual spreading}.

{\it Subcase 2.2: $\bar{\beta} \geqslant  B(\tilde{\beta})$}. In the special case where $\bar{\beta}
> B(\tilde{\beta})$ we have $\bar{\beta}-\bar{c} >\overline{r(t;\beta)}$, that is, the back moves rightward faster than the front. So the solution is suppressed by its back, and $u\to 0$ uniformly. This is called {\it vanishing} phenomenon.
In the critical case where $\bar{\beta} =  B(\tilde{\beta})$ we have the same conclusion by a
more delicate approach.

Finally we remark that, if the equation in \eqref{p} is replaced
by a more general one:
$$
u_t = d(t) u_{xx} -\beta(t) u_x + f(t,u),
\quad g(t) <x<h(t),\ t>0,
$$
where $d\in \PP_+$, then by taking a new time
variable $\tau = D(t):=\int_0^t  d(s)ds$, we see that the function $v(\tau, x) := u(D^{-1}(\tau),x)$
solves a problem like \eqref{p}. In particular, the coefficient of the diffusion term in the equation of
$v$ is $1$. Therefore the argument in this paper applies for such a general equation.

This paper is organized as the following. In Section 2 we present our main results.
In Section 3 we construct several kinds of traveling waves and give an equivalent
description for the set $\{B(\theta) +\theta : \theta \in \PP^0 \}$ of the second critical functions.
In Section 4 we
study the long time behavior for the solutions and prove Theorems 2.1, 2.2 and 2.3.
In Section 5 we consider the asymptotic profiles for (virtual) spreading solutions
and prove Theorem 2.4.

\section{Main Results}\label{semar}

Throughout this paper we choose the initial data $u_0$ from the
following set.
\begin{equation}\label{def:X}
\mathscr {X}(h_0):= \Big\{ \phi \in C^2 ([-h_0,h_0]) :
\phi(-h_0)= \phi (h_0)=0,\; \phi(x) \geqslant ,\not\equiv 0 \
\mbox{in } (-h_0,h_0).\Big\}
\end{equation}
where $h_0 >0$ is a real number. By a similar argument as in \cite{DuLin,DuLou}, one can show that, for any $h_0 >0$
and any initial data $u_0\in\mathscr{X}(h_0)$, the problem \eqref{p} has a time-global
solution $(u,g,h)$, with $u\in C^{1+\nu/2,2+\nu }((0,\infty) \times[g(t),h(t)])$ and $g,\, h\in C^{1+\nu/2}((0,\infty))$
for the number $\nu$ in $\HH$. Moreover, it follows from the maximum principle
that, when $t>0$, the solution $u$ is positive in $(g(t),h(t))$,
$u_x(t,g(t))>0$ and $u_x(t,h(t))<0$, thus $g'(t)<0<h'(t)$ for all
$t>0$. Denote
$$
g_{\infty}:=\lim_{t\to\infty}g(t),\quad h_{\infty}:=
\lim_{t\to\infty}h(t)\quad \mbox{and} \quad I_{\infty}:=(g_{\infty},h_{\infty}).
$$
Now we list some possible situations on the asymptotic behavior of the solutions to \eqref{p}.
\begin{itemize}
\item {\it spreading} : $I_\infty =\R$ and
$\lim_{t\to\infty} \ [u(t,\cdot)-P(t)] =0$
locally uniformly in $\R$;

\item {\it vanishing} :  $I_\infty$ is a bounded interval and
$\lim_{t\to\infty} \|u(t,\cdot)\|_{L^\infty ([g(t),h(t)])}=0$;

\item {\it virtual spreading} : $g_\infty>-\infty$, $h_\infty=+\infty$, $\lim_{t\to\infty}u(t,\cdot)=0$
locally uniformly in $(g_\infty, \infty)$ and, for some $c_1 >0$, $\lim_{t\to\infty} \ [u(t,\cdot+ c_1 t) -P(t)] =0$
locally uniformly in $\R$.
\end{itemize}

Our first main result deals with the small advection case.

\begin{thm}\label{thm:small beta}
Assume $\HH$ and $0\leqslant  \bar{\beta}<\bar{c}$. Let $(u,g,h)$ be a time-global solution of \eqref{p} with $u_0=\sigma\phi$ for some
$\phi\in \mathscr {X}(h_0)$ and $\sigma\geqslant 0$. Then there exists $\sigma^*=\sigma^*
(h_0, \phi, \beta)\in [0,\infty ]$ such that {\rm vanishing}
happens when $\sigma\in [0,\sigma^*]$ and {\rm spreading} happens when $\sigma>\sigma^*$.
\end{thm}

When the advection is a medium-sized one we have the following trichotomy result.

\begin{thm}\label{thm:middle beta}
Assume $\HH$ and $\bar{c} \leqslant \bar{\beta}<B(\tilde{\beta})$. 
Let $(u,g,h)$ be a time-global solution of \eqref{p} with $u_0=\sigma \phi$ for some
$\phi\in \mathscr {X}(h_0)$ and $\sigma\geqslant 0$. Then there exist $\sigma_*=\sigma_*(h_0, \phi,\beta)$
and $\sigma^*=\sigma^* (h_0, \phi, \beta)$ with $0< \sigma_*\leqslant  \sigma^* \leqslant  \infty$
such that
\begin{itemize}
\item[(i)] {\rm vanishing} happens when $\sigma\in [0,\sigma_*)$;

\item[(ii)] {\rm virtual spreading} happens when $\sigma > \sigma^*$;

\item[(iii)] in the {\rm transition} case $\sigma \in [\sigma_*, \sigma^*]$:
$g_\infty >-\infty$, $h_\infty =\infty$, $u(t,\cdot) \to 0$ as $t\to \infty$ locally uniformly
in $(g_\infty, \infty)$, and virtual spreading does not happen.
\end{itemize}
\end{thm}

\noindent
In the transition case, it is clear by the values of $g_\infty$ and $h_\infty$ that neither vanishing nor
spreading happens. In fact, in the homogeneous case (that is, all the coefficients in \eqref{p} are
independent of $t$), \cite{GLZ} proved that any transition solution converges to $\mathcal{V}
(x-(\beta-\bar{c})t +o(t))$, where $\mathcal{V}(x -(\beta-\bar{c})t)$ is a {\it tadpole-like traveling semi-wave}
whose profile has a big head and a boundary on the right side and an infinite long tail on the left side.
We guess that in the transition case in our theorem, the solution also converges to such a traveling wave
$\wh{\mathcal{V}}$, with time-periodic tadpole-like profile and with average speed $\bar{\beta}-\bar{c}$.
This problem remains open now. The main difficulty is to prove the existence of $\wh{\mathcal{V}}$ itself
whose profile is time-periodic in $t$ and non-monotone in the space variable.

When the advection is large, the long time behavior of the solutions is rather simple under an
additional condition:
$$
(H_1)\  \hskip 30mm   \alpha(t):= P(t) \cdot f_u(t, P(t)) -f(t,P(t)) < 0 \mbox{ for } t\in [0,T],
 \hskip 50mm
$$
where $P(t)$ is the periodic solution of $u_t = f(t,u)$.
A typical example of such $f$ is $f(t,u) = u(a(t) - b(t) u)$ with positive and periodic functions $a$ and $b$.
When $f=f(u)$ is a homogeneous Fisher-KPP type of nonlinearity with zeros $0$ and $1$, the condition $(H_1)$
reduces to $f'(1)<0$.

\begin{thm}\label{thm:large beta}
Assume $\HH$, $\HHH$ and $\bar{\beta} \geqslant B(\tilde{\beta})$. Let $(u,g,h)$ be a time-global
solution of \eqref{p} with initial data $u_0\in\mathscr {X}(h_0)$. Then vanishing happens.
\end{thm}

\medskip

Finally we consider the asymptotic profiles and speeds for the solutions when (virtual) spreading happens
as in Theorems \ref{thm:small beta} and \ref{thm:middle beta}. Roughly speaking, near the right boundary
$x=h(t)$, both the spreading solutions and the virtual spreading ones can be characterized by the periodic
rightward traveling semi-wave $U(t, R(t)-x)$. Near the left boundary $x=g(t)$, however, the behavior of a
spreading solution is different from that of a virtual spreading one. The former can be characterized near the left boundary by a
{\it periodic leftward traveling semi-wave} $\wh{U} (t,  x+L(t))$, where $L(t):=
\int_0^t l(s;\beta) ds$ for some $l\in \PP_+$ and $(l,\,\wh{U})$ solves the following problem
\begin{equation}\label{p-sw-k-left}
\left\{
\begin{array}{ll}
U_t = U_{zz} - [\beta +l] U_z + f(t,U), &  t\in [0,T],\ z>0,\\
U(t,0)=0,\ \ U(t, \infty)= P(t),  & t\in [0,T],\\
U(0,z) = U(T,z),\ U_z(t,z)>0 & t\in [0,T],\ z \geqslant 0,\\
l(t) = \mu(t) U_z(t,0), & t\in [0,T].
\end{array}
\right.
\end{equation}

\begin{thm}\label{thm:profile of spreading sol}
Assume $\HH$ and $\HHH$. Assume also that spreading happens for a solution of \eqref{p} as
in Theorem \ref{thm:small beta} or virtual spreading happens as in Theorem \ref{thm:middle beta}.
 Let $(r, U)$ be the unique solution of \eqref{p-sw-k}.
\begin{itemize}
  \item [(i)] When $0\leqslant  \bar{\beta} < \bar{c}$, let $(l,\wh{U})$ be the unique solution of \eqref{p-sw-k-left}.
  Then there exist $H_1, G_1\in\R$ such that
\begin{equation}\label{right spreading speed}
\lim\limits_{t\to\infty}[h(t)- R(t)] = H_1 ,\quad \lim\limits_{t\to\infty} [h'(t) -r(t;\beta)] = 0,
\end{equation}
\begin{equation}\label{left spreading speed}
\lim\limits_{t\to\infty}[g(t) + L(t)] = G_1 ,\quad \lim\limits_{t\to\infty} [g'(t)+ l(t;\beta)] = 0,
\end{equation}
where $R$ and $L$ are the integrals of $r$ and $l$, respectively. In addition, if we extend $U$, $\wh{U}$ to
be zero outside their supports we have
\begin{equation}\label{profile convergence 1}
\lim\limits_{t\to\infty} \left\| u(t,\cdot)- U(t,  R(t)+ H_1-\cdot) \right\|_{L^\infty ( [0, h(t)])}=0,\\
\end{equation}
and
\begin{equation}\label{profile convergence 1-left}
\lim\limits_{t\to\infty} \left\| u(t,\cdot)- \wh{U} \big(t, \cdot +L(t) - G_1 \big)\right\|_{L^\infty ([g(t), 0])} =0.
\end{equation}
  \item [(ii)] When $\bar{c} \leqslant  \bar{\beta} < B(\tilde{\beta})$, then \eqref{right spreading speed} holds and
\begin{equation}\label{profile convergence }
\lim\limits_{t\to\infty} \left\| u(t,\cdot)- U(t,  R(t)+ H_1-\cdot) \right\|_{L^\infty ( [c_1 t, h(t)])}=0,
\end{equation}
for any $c_1$ satisfying $\bar{\beta} -\bar{c} < c_1 <  \ov{r(t;\beta)}$.
\end{itemize}
\end{thm}

As we have mentioned above,  \cite{DGP, DMZ, GLZ, LLS1, Wang} etc. considered some special cases of \eqref{p}.
Compared with their results, there are some new difficulties and breakthroughs in our approach caused by
the temporal inhomogeneity and the advection.

1.{\it Problems in advective environments}. In \cite{DGP, LLS1, Wang} the authors considered
the problem \eqref{p} without advection (with $f(t,u)$ being periodic or almost periodic in $t$).
They all presented similar results as our Theorem \ref{thm:small beta}.
But the analogue of Theorems \ref{thm:middle beta} and \ref{thm:large beta} is not obtained
since the advection is not involved in their problems.

2.{\it Partition of $\beta(t)$}. In \cite{GLZ}, the authors considered the homogeneous version of
\eqref{p}, which is a problem with advection. Two critical values $\bar{c}$ and $\beta^*$ with $\beta^* >\bar{c}>0$
were found such that, $\bar{c}$ is the first partition point of $\beta$ separating the small and medium-sized
advections, and $ \beta^*$ is the second partition point of $\beta$ separating the medium-sized and large
advections. In our problem \eqref{p}, however, $\beta(t)$ is a periodic function, whose partition is
much more complicated since this is related not only to the ``size" $\bar{\beta}$ but also to  the
``shape" $\tilde{\beta}$. In particular, the description for the second critical functions
separating the medium-sized and large advections is far from trivial,
it is based on some additional properties like: $r(t;\beta)$ and $\beta - r(t;\beta)$
are increasing functions of $\beta$ (cf. Section 1 and Section 3).

3.{\it Construction of periodic solutions and periodic traveling waves}.
In our approach we need several kinds of periodic traveling waves. For homogeneous problems,
these traveling waves can be obtained by a simple phase plane analysis (cf. \cite{GLZ}).
But for our inhomogeneous problem \eqref{p}, we prove the existence of such traveling waves by
a totally different approach inspired by \cite{DGP} and \cite{P}.

4.{\it Precise estimate for the spreading speed and the asymptotic profile for a (virtual)
spreading solution to time-periodic problems}. When (virtual) spreading happens for a solution,
the asymptotic profile and the asymptotic speed of the solution are interesting problems in applied fields.
As far as we know, it is the first time that our Theorem \ref{thm:profile of spreading sol} gives out a sharp description
for temporal inhomogeneous case. Before it, only a rough estimate:
$\lim_{t\to \infty} h(t)/t = k$ (for some $k>0$) was obtained.

\section{Preliminaries}\label{prelis}

\subsection{Positive solutions on bounded and unbounded intervals}
In this part we present several kinds of positive solutions to \eqref{p}$_1$, which will be used to construct
traveling waves in the next subsection. Recall that we write $a(t) := f_u(t,0)$ throughout this paper.

First, for any given $k\in \PP$ and $\ell >0$, we consider the following $T$-periodic eigenvalue problem:
\begin{equation}\label{eigen-p}
\left\{
 \begin{array}{ll}
 \mathcal{L} \varphi := \varphi_t - \varphi_{zz} -k (t) \varphi_z -a(t)\varphi =\lambda \varphi, & t\in [0,T],\ z\in (0,\ell),\\
 \varphi(t,0) = \varphi(t,\ell)=0, & t\in [0,T],\\
 \varphi(0,z) = \varphi(T,z), & z\in [0,\ell].
 \end{array}
 \right.
\end{equation}
On the sign of the principal eigenvalue $\lambda_1 (\ell)$ there is a well known result.

\begin{lem}[\cite{DGP, P}]\label{lem:1eigenvalue}
If $k\in \PP$ satisfies $|\bar{k}|<\bar{c} = 2\sqrt{\bar{a}}$, then there exists
$\ell^* = \ell^*(k,a)>0$ such that the principal eigenvalue $\lambda_1(\ell)$ of \eqref{eigen-p} is negative (resp. $0$,
or positive) when $\ell>\ell^*$ (resp. $\ell =\ell^*$, or $\ell <\ell^*$).

If $k\in \PP$ satisfies $|\bar{k}| \geqslant \bar{c}$ then $\lambda_1(\ell) >0$ for any $\ell>0$.
\end{lem}

Using this lemma and using the standard method of lower and upper solutions one can easily obtain the $T$-periodic solutions of
the following problem:
\begin{equation}\label{p DD0}
\left\{
 \begin{array}{ll}
 v_t = v_{zz} + k(t) v_z + f(t,v), & t\in [0,T],\ z\in (0,\ell),\\
 v(t,0) = v(t,\ell)=0, & t\in [0,T],\\
 v(0,z) = v(T,z), & z\in [0,\ell].
 \end{array}
 \right.
\end{equation}

\begin{lem}[\cite{P, Nad}]\label{lem:positive sol DD0}
Assume that $k\in\PP$ satisfies $|\bar{k}|<\bar{c}$. Then there exists a real number $\ell^*
:= \ell^* (k,a)$  such that, when $\ell >\ell^*$ the problem \eqref{p DD0} has a unique
solution $v=U_0(t,z; k,\ell)$ which satisfies $0<U_0(t,z;k,\ell)< P(t)$ in $[0,T] \times (0, \ell)$ and $(U_0)_z(t, 0; k, \ell)>0$ for
$t\in [0,T]$. Moreover, $U_0(t,z;k,\ell)$ is strictly increasing in $\ell$ and
 $U_0(t,z + \ell/2 ; k,\ell) \to P(t)$ as $\ell \to \infty$ in $L^\infty_{loc}(\R)$ topology;
when $\ell \leqslant  \ell^*$, the problem \eqref{p DD0} has only zero solution.

Assume that $k\in \PP$ satisfies $|\bar{k}|\geqslant \bar{c}$. Then for any $\ell >0$, the problem \eqref{p DD0} has only zero solution.
\end{lem}

Next we consider the problem \eqref{p DD0} with different boundary condition at $z=\ell$, that is,
\begin{equation}\label{p DD1}
\left\{
 \begin{array}{ll}
 v_t = v_{zz} + k(t) v_z + f(t,v), & t\in [0,T],\ z\in (0,\ell),\\
 v(t,0) = 0,\ v(t,\ell)= P^0:= \max_{0\leqslant  t\leqslant  T} P(t), & t\in [0,T],\\
 v(0,z) = v(T,z), & z\in [0,\ell].
 \end{array}
 \right.
\end{equation}

\begin{lem}\label{lem:positive sol DD1}
For any $k\in \PP$  and any $\ell >0$, the problem \eqref{p DD1} has a maximal solution $v= U_1 (t,z;k,\ell)$,
which is strictly increasing in both $z\in [0,\ell]$ and $k\in \PP$, strictly decreasing in $\ell >0$.

If $k\in \PP$ satisfies $\bar{k}>-\bar{c}$, then there exists $\delta>0$ independent of $\ell$ such that $(U_1)_z (t,0;k,\ell)
\geqslant \delta$ for $t\in [0,T]$.
\end{lem}

\begin{proof}
Consider the equation and the boundary conditions in \eqref{p DD1} with initial data $v(0,z):= P^0 \cdot \chi_{[0,\ell]} (z)$,
where $\chi_{[0,\ell]} (z)$ is the characteristic function on the interval $[0,\ell]$. This initial-boundary value problem has a
unique solution $v(t,z; k,\ell)$. Using the maximum principle we see that $v(t,z;k,\ell)$ is strictly
increasing in $z\in [0,\ell]$ and in $k\in \PP$, strictly decreasing in $\ell >0$ and $v(t,z;k,\ell)\leqslant  P^0$.
Using the zero number argument in a similar way as in the proof of \cite[Theorem 1]{BPS}\footnote{In \cite{BPS},
$f\in C^2$ is assumed and $v\to U_1$ is taken in $H^2 ([0,\ell])$. Note
that for our problem \eqref{p DD1}, the assumption for $f$ and $k$ is sufficient to guarantee that the omega limit set of
$v(t,\cdot)$ in the topology $C^2 ([0,\ell])$ is not empty, and then a similar zero number argument as in
\cite{BPS} gives the convergence $v\to U_1$ in $C^2 ([0,\ell])$.
Moreover, the zero number properties we used here are those in Angenent \cite{A}, where the coefficient
$k(t)$ of $v_z$ is assumed to be in $W^{1,\infty}([0,T])$. We remark that for our problem, the condition
$k\in C^{\nu/2}([0,T])$ is sufficient to proceed the zero number argument.
In fact, denote $K(t):= \int_0^t k(s)ds$, $y:= z+K(t)$ and $w(t,y):= v(t, y-K(t))$, then the equation of $v$ is
converted into $w_t = w_{yy} + f(t,w)$. This equation has no the first order term, though it is considered in a
moving frame $K(t) < y < K(t) +\ell$, the zero number properties as in \cite{A} remain valid (cf. \cite{DuLouZ,GLZ}).
In Subsection 5.2, we use the zero number properties in the same way.}
one can show that $\|v(t,\cdot ;k,\ell) - U_1(t,\cdot ;k,\ell)\|_{C^2  ([0,\ell])} \to 0$ as $t\to \infty$,
where $U_1 (t,z;k,\ell)\in C^{1+\nu/2, 2+\nu} ([0,T]\times [0,\ell])$ is a time periodic solution of \eqref{p DD1}.
By the maximum principle again, we see that $U_1$ has the same monotonic properties as $v$ in $z,\ k$ and $\ell$.

We now assume $\bar{k}>-\bar{c}$ and show the existence of the positive lower bound for $(U_1)_z(t,0;k,\ell)$.

\smallskip

{\it Case 1}. $|\bar{k}| <\bar{c}$. In this case, by Lemma \ref{lem:positive sol DD0}, there exists a positive
number $\ell^*(k, a)$ such that the problem \eqref{p DD0} with $\ell= \ell_1 := \ell^*(k,a) +1$ has a time periodic
solution $U_0(t, z; k, \ell_1)$ which satisfies $\delta_1 := \min_{0\leqslant  t\leqslant  T} (U_0)_z (t ,0; k, \ell_1 ) >0$.
By the comparison principle we have $v(t,z;k,\ell)\geqslant U_0(t,z;k,\ell_1 )$ for all $t>0$ and $z\in [0, \min\{\ell, \ell_1\}]$.
Taking limit as $t\to \infty$ we have $U_1 (t,z;k,\ell)\geqslant U_0(t,z;k,\ell_1 )$ for all $t>0$ and $z\in [0, \min\{\ell, \ell_1\}]$,
and so $(U_1)_z (t,0;k,\ell) \geqslant \delta_1 $ for $t\in [0,T]$.

\smallskip

{\it Case 2}. $\bar{k}\geqslant \bar{c}$. Set $\tilde{k} := k-\bar{k}$, then by Lemma \ref{lem:positive sol DD0} again,
there exists a positive
number $\ell^*(\tilde{k}, a)$ such that the problem \eqref{p DD0} with $k$ replaced by $\tilde{k}$ and $\ell= \ell_2 := \ell^*(\tilde{k},a) +1$
has a time periodic solution $U_0(t, z; \tilde{k}, \ell_2)$, which satisfies $\delta_2 := \min_{0\leqslant  t\leqslant  T}
(U_0)_z (t ,0;\tilde{k}, \ell_2 ) >0$. Let $\ell$ be the width of the interval in the problem \eqref{p DD1} and set
$\tau_0 := \ell/ \bar{k}$, then for any $\tau >\tau_0$, the function $\underline{v}(t,z) :=
U_0(t, z+\bar{k}(t-\tau) ; \tilde{k}, \ell_2)$ solves the following problem:
$$
\left\{
 \begin{array}{l}
 v_t = v_{zz} +k v_z + f(t,z), \quad  \bar{k} (\tau -t) < z < \ell_2 + \bar{k}(\tau -t),\ t>0,\\
 v(t, \bar{k}(\tau -t) )= v(t,  \ell_2 + \bar{k}(\tau -t) )=0, \quad t>0.
 \end{array}
 \right.
$$
We will compare $v(t,z;k,\ell)$ and $\underline{v}(t,z)$. When $t\in [0,\tau -\tau_0)$, there is no need to compare
them since their spatial domains have no intersection. When $t\in (\tau -\tau_0, \tau]$ we have $0\leqslant \bar{k}(\tau -t) <\ell$,
and so they have common spatial domain $J(t):= [\bar{k}(\tau -t), \min\{\ell_2 +\bar{k}(\tau -t),\; \ell\}]$. By
the comparison principle we have  $v(t,z;k,\ell)\geqslant \underline{v}(t,z)$ for all $z\in J(t)$ and $t\in (\tau -\tau_0, \tau]$.
In particular, at $t=\tau$ we have
$v(\tau ,z;k,\ell)\geqslant U_0(\tau,z;\tilde{k},\ell_2 )$ in $z\in J(\tau)$.
Therefore, $v_z(\tau,0; k,\ell) \geqslant (U_0)_z(\tau,0;\tilde{k},\ell_2) \geqslant \delta_2 $. Since $\tau >\tau_0$ is arbitrary we have $(U_1)_z(t,0;k,\ell)\geqslant \delta_2$
for $t\in [0,T]$.
\end{proof}

Finally, let us consider the problem on the half line
\begin{equation}\label{p Dinfty}
\left\{
\begin{array}{ll}
v_t = v_{zz} + k(t) v_z + f(t,v), &  t\in [0,T],\ z>0,\\
v(t,0)=0,  & t\in [0,T],\\
v(0,z) = v(T,z), & z \geqslant 0.
\end{array}
\right.
\end{equation}

\begin{lem}\label{lem:p tsw k half}
For each $k\in \PP$, the problem \eqref{p Dinfty} has a bounded and nonnegative solution $U(t,z;k)$.
Moreover,
\begin{itemize}
\item[(i)] if $\bar{k} > - \bar{c}$, then
\begin{equation}\label{U increase to P(t)}
U_z(t,z;k)>0 \mbox{ in } [0,T]\times [0,\infty), \quad U(t,z;k)- P(t)\to 0 \mbox{ as } z\to \infty;
\end{equation}
$U(t,z;k)$ is the unique solution of \eqref{p Dinfty} satisfying \eqref{U increase to P(t)};
$U_z(t,0;k)$ has a positive lower bound $\delta$ (independent of $t$), and it is strictly increasing
in $k$: $U_z(t,0;k_1)< U_z(t,0;k_2)$ for $k_1, k_2\in \PP$ satisfying $k_1\leqslant, \not\equiv k_2$ and $\ov{k_1}, \ov{k_2} >-\bar{c}$;
$U_z(t,0;k) $ is continuous in $k$ in the sense that, for $\{k_1, k_2, \cdots\} \subset \PP$
satisfying $\ov{k_i} >-\bar{c}\ (i=1,2,\cdots)$,
$U_z(t,0;k_n)\to U_z(t,0; k)$ in $C^{\nu/2}([0,T])$ if $k_n \to k$ in $C^{\nu/2} ([0,T])$.

\item[(ii)] when $\bar{k}\leqslant -\bar{c}$, \eqref{p Dinfty} has only trivial solution $0$.
\end{itemize}
\end{lem}

\begin{proof}
Let $U_1 (t,z;k,\ell) $ be the solution of \eqref{p DD1} obtained
in the previous lemma. Since it is decreasing in $\ell$, by taking limit as $\ell \to \infty$ we see that $U_1 (t,z;k,\ell)$ converges to some function $U(t,z;k)$, which is non-decreasing in $z$ and in $k$ since $U_1$ is so. By standard regularity argument, $U$ is a classical solution of \eqref{p Dinfty}.

(i) In case $\bar{k}>-\bar{c}$, Lemma \ref{lem:positive sol DD1} implies that $U_z(t,0;k)\geqslant\delta >0$. Using the strong maximum principle to $U_z$ we conclude that $U_z(t,z;k)>0$ in $[0,T]\times [0,\infty)$. Thus $P_1(t):= \lim_{z\to \infty} U(t,z;k)$ exists.
In a similar way as in the proof of \cite[Proposition 2.1]{DGP} one can show that $P_1(t)$ is nothing but the positive periodic solution $P(t)$ of $u_t =f(t,u)$. The uniqueness of $U(t,z;k)$ and its continuous dependence in $k$ can be proved in a similar way as \cite[Theorems 2.4 and 2.5]{DGP}.
Since $U(t,z;k)$ is non-decreasing in $k$ we have $U_z(t,0;k_1)\leqslant U_z(t,0;k_2)$
when $k_1 \leqslant k_2$. The strict inequality $U_z(t,0;k_1)<  U_z(t,0;k_2)$ follows
from the Hopf lemma and the assumption $k_1 \leqslant,\not\equiv k_2$.

(ii). The conclusion can be proved in a similar way as in the
proof of \cite[Proposition 2.3]{DGP}.
\end{proof}

Let $U_0(t,z;k,\ell)$ and $U(t,z;k)$ be the solutions obtained in the above lemmas, denote
$$
A_0 [k,\ell](t) := \mu(t) (U_0)_z(t,0;k,\ell),
\quad A [k] (t):= \mu(t) U_z (t,0;k),
$$
where $\mu(t)$ is the function in the Stefan condition in \eqref{p}.
From Lemma \ref{lem:p tsw k half} we see that $A[k](t)$ is strictly increasing in $k\in\PP$ when $\bar{k}>-\bar{c}$.
We now show the convergence $A_0 [k,\ell]\to A[k]$ as $\ell \to \infty$.

\begin{lem}\label{lem:A0 ell to A}
Assume that $k\in \PP$ satisfies $|\bar{k}|<\bar{c}$. Then
$A_0 [k,\ell] (t)\to A[k](t)$ in $L^\infty ([0,T])$ norm as
$\ell \to \infty$.
\end{lem}

\begin{proof}
It follows from Lemma \ref{lem:positive sol DD0} that the problem \eqref{p DD0} admits a unique solution $U_0(t, z; k, \ell)$ when
$\ell > \ell^*(k, a)$, and $U_0$ is strictly increasing in $\ell$.
Hence $U_0(t, z; k,\ell)$ converges as $\ell \to \infty$ to some function $U_\infty (t,z)$, in the topology of $L^\infty_{loc}([0,T]
\times [0,\infty))$. (The convergence is also true in the topology of $C^{1,2}_{loc} ([0,T]\times [0,\infty))$ by the parabolic estimates.)
Moreover, $U_\infty (t,\infty) =P(t)$ by Lemma \ref{lem:positive sol DD0}. Then a similar argument as
in the proof of \cite[Proposition 2.1]{DGP} shows that $U_\infty $ is nothing but the unique positive solution $U(t,z;k)$ of \eqref{p Dinfty} as in Lemma
\ref{lem:p tsw k half}(i). Consequently, $\|(U_0)_z(t,z;k,\ell) - U_z(t,z;k)\|_{C([0,T] \times [0,1])} \to 0 $ as $\ell \to \infty$.
This proves the lemma.
\end{proof}

\subsection{Periodic traveling waves}
Based on the results in the previous subsection, we now construct several kinds of
traveling waves of \eqref{p}$_1$.

\medskip
\noindent
{\bf (I). Periodic rightward traveling semi-waves}.
In this part we construct a traveling semi-wave which is periodic in time and is used to
characterize spreading solutions near the right boundaries. First we present a lemma.

\begin{lem}\label{lem subso}
Assume $\beta\in \PP$ satisfies $\bar{\beta}\geqslant 0$. If $r_i\in \PP_+$ satisfies $\ov{r_i} < \bar{\beta} +\bar{c}$ ($i=1, 2, 3$),
$r_1 \leqslant A[\beta-r_1],\ A[\beta -r_2] =r_2$ and $A[\beta - r_3] \leqslant r_3$,
then $r_1 \leqslant r_2 \leqslant r_3$.
\end{lem}

\begin{proof}
We only prove $r_1\leqslant r_2$ since $r_2\leqslant r_3$ can be proved in a similar way.
If $r_1 \equiv r_2$, then there is nothing left to prove. In what follows we assume
$r_1\not\equiv r_2$ and prove $r_1 \leqslant r_2$ by using a similar idea as in
the proof of \cite[Theorem 2.5]{DGP}.

\medskip
{\bf Step 1}. First we show that $r_1 \geqslant r_2$ is impossible.
Otherwise, $\beta -r_1 \leqslant, \not\equiv \beta -r_2$ and so
$r_1 \leqslant A[\beta -r_1] < A[\beta -r_2] =r_2 $ by the monotonicity of $A$, contradicting the assumption
$r_1 \geqslant r_2$. Therefore $r_1\not\geqslant r_2$ and so
\begin{equation}\label{r1<r2}
r_1 (s_0) < r_2 (s_0) \mbox{  for some } s_0\in [0,T).
\end{equation}

Now we prove $r_1 \leqslant r_2$ and suppose by way of contradiction that $r_1(s_*)>r_2(s_*)$
for some $s_*\in [0,T)$. Then
$$
s^* := \sup \{ \tau \in (s_*, s_* +T)  : r_1(t) > r_2 (t) \mbox{ for } t\in [s_*, \tau)\}
$$
is well-defined by \eqref{r1<r2}. So we have
\begin{equation}\label{r1>r2s*}
r_1(t)>r_2 (t) \mbox{ for } t\in [s_*, s^*),\quad r_1(s^*) =r_2(s^*).
\end{equation}
Denote
$$
R_1(t):= \int_{s_*}^t r_1(t) dt,\quad R_2(t):= \int_{s_*}^t r_2(t) dt + X \mbox{ with }
X:= \int_{s_*}^{s^*} [r_1-r_2]dt.
$$
Then $R_1(t) <R_2(t)$ for $t\in [s_*, s^*)$ and $ R_1(s^*) =R_2(s^*)$ (denoted it by $x^*$).
For $i=1$ or $2$, the problem \eqref{p Dinfty} with $k=\beta -r_i$ (note that $\bar{\beta} - \ov{r_i} >-\bar{c}$ by our assumption)
has a maximal bounded solution $U(t,z; \beta-r_i)$, which is positive for $z>0$.
To derive a contradiction, let us consider the functions $u_i (t,x) :=U(t, R_i(t)-x; \beta -r_i)$.
It is easy to check that
\begin{equation}\label{VV1}
\left\{
\begin{array}{ll}
u_{it}  = u_{ixx} -\beta(t) u_{ix} + f(t, u_i), &  x< R_i(t),\ t\in[s_*, s^*],\\
u_i (t,R_i(t))=0,\ \ u_i (t, -\infty)= P(t),  & t\in [s_*, s^*],
\end{array}
\right.
\end{equation}
and, for $t\in[s_*, s^*]$,
\begin{equation}\label{rVrV}
 r_1(t)\leqslant -\mu(t) u_{1x} (t,R_1(t)) = A[\beta -r_1](t),\quad  r_2(t)=-\mu(t) u_{2x} (t,R_2(t)) =A[\beta -r_2](t).
\end{equation}
Set $W(t,x):= u_2 (t,x)-u_1(t,x)$ for $(t,x)\in \Omega :=\{(t,x)  :   x < R_1(t),\ t\in [s_*, s^*]\}$.
In the next step we will prove a claim: $W(t,x)>0$ in $\Omega\backslash \{(s^*, x^*)\}$ and $W_x (s^*, x^*)<0$.
Once this claim is proved, we have $u_{1x} (s^*,x^*) > u_{2x} (s^*,x^*)$. Combining with \eqref{rVrV} we derive $r_1(s^*) < r_2(s^*)$,
contradicting \eqref{r1>r2s*}.  This completes the proof for $r_1 \leqslant r_2$.

\medskip
{\bf Step 2}. To prove the claim: $W(t,x)>0$ over $\Omega\backslash \{(s^*, x^*)\}$ and $W_x (s^*, x^*)<0$.

For each fixed $t\in [s_*, s^*]$, since $R_1(t)\leqslant R_2(t)$ and since $u_i (t,-\infty)=P(t)$,
we have $u_2(t, x)> \rho u_1(t, x)$ for $x<R_1(t)$ provided $\rho >0$ is sufficiently small. So
$$
\rho^* :=\sup \{\rho>0  : u_2(t, x)> \rho u_1(t, x) \mbox{ for } x<R_1(t),\ t\in [s_*, s^*]\}
$$
is well-defined, and $0<\rho^* \leqslant 1$. Using the Fisher-KPP property in $\HH$, one can show that
$$
(\rho^* u_1)_t - (\rho^* u_1)_{xx} + \beta(t) (\rho^* u_1)_x -f(t, \rho^* u_1) \leqslant 0,\quad x<R_1(t),\ t\in [s_*, s^*],
$$
that is, $\rho^* u_1$ is a lower solution of \eqref{VV1}. By the strong comparison principle and
the Hopf lemma we have
$$
W^*(t,x) := u_2 (t,x) - \rho^* u_1(t,x) >0 \mbox{ in } \Omega\backslash \{(s^*, x^*)\}\quad
\mbox{ and } \quad W^*_x(s^*, x^*) <0.
$$

The claim is proved if $\rho^* =1$. Suppose by way of contradiction that $0<\rho^*<1$.
Then by the definition of $\rho^*$, for any sequence of positive numbers $\varepsilon_n \to0$, there exists
$(t_n,x_n)$ with $t_n\in [s_*, s^*]$ and $x_n < R_1(t_n)$ such that
\begin{equation}\label{V2V1}
u_2(t_n, x_n) < (\rho^* +\varepsilon_n ) u_1(t_n, x_n) \mbox{ for }  n\geqslant 1.
\end{equation}
By passing to a subsequence, we may assume that $t_n\to \hat{t} \in[s_*, s^*]$.
We show that $x_n$ has a lower bound that is independent of $n$. Otherwise, by passing to a
subsequence we may assume that $x_n\to -\infty$ as $n\to\infty$. Then
$R_2(t_n) -x_n\geqslant R_1(t_n)-x_n\to\infty$ and hence $u_i(t_n,x_n)\to P(\hat{t})$ as $n\to\infty$.
It follows from \eqref{V2V1} that $P(\hat{t})\leqslant  \rho^* P(\hat{t})$, contradicting our assumption $\rho^* \in (0,1)$.
This proves the boundedness of $x_n$, and so we may assume without loss of generality
that $x_n\to\hat{x}$ as $n\to \infty$. This, combining with \eqref{V2V1} again, leads to
$W^*(\hat{t}, \hat{x})\leqslant 0$. Since $W^*>0$ in $\Omega\backslash \{(s^*, x^*)\}$,
we necessarily have $(\hat{t}, \hat{x}) = (s^*, x^*)$, $W^*(\hat{t}, R_1(\hat{t}))=0$ and
$W^*_x(\hat{t}, R_1(\hat{t}))<0$. By continuity we can find positive constants $\epsilon_1$ and $\delta_1$ such that
\[
W_x^* (t,x)<- \delta_1 \mbox{ for }  x\in [R_1(t)-\epsilon_1, R_1(t)],\ t\in[s^* -\epsilon_1, s^*].
\]
This implies that
\[
W^*(t_n, x_n)\geqslant  W^*(t_n, x_n) - W^*(t_n, R_1(t_n)) \geqslant \delta_1 [R_1(t_n)- x_n]\ \mbox{ for all large } n.
\]
On the other hand, it follows from $u_1(t, R_1(t))=0$ and $- u_{1x} (t,x)= U_z (t,R_1(t)-x; \beta -r_1) < C$
(for some $C>0$) that
\[
u_1(t_n, x_n)=  u_1(t_n, x_n) - u_1(t_n, R_1(t_n)) \leqslant  C [R_1(t_n)-x_n] \mbox{ for all large } n.
\]
Thus for large $n$ we have
\[
u_2(t_n, x_n)\geqslant \rho^* u_1(t_n, x_n)+\delta_1[R_1( t_n)- x_n] \geqslant \Big(\rho^*+\frac{\delta_1}{C}\Big) u_1(t_n, x_n),
\]
which contradicts \eqref{V2V1}. This proves $\rho^*=1$ and so the claim is true.
\end{proof}

\begin{remark}\label{rem:unique r}
\rm
There are two simple consequences following from the previous lemma, one is $r_1 \leqslant, \not\equiv r_2$ when
$r_1 \leqslant,\not\equiv A[\beta-r_1]$ and $A[\beta -r_2] =r_2$; another one is $r_1 \equiv r_2$ when $r_i = A[\beta-r_i]\ (i=1,2)$,
which implies that, for each $\beta\in \PP$ with $\bar{\beta}\geqslant 0$, the equation $r= A[\beta -r]$ has at most
one solution $r$.
\end{remark}

On the existence and the properties of the solution of $r=A[\beta -r]$ we have the following key result.

\begin{prop}\label{prop:tsw right}
Assume $\beta \in \PP$ satisfies $\bar{\beta}\geqslant 0$. Then there exists a unique function $r (t;\beta)\in \PP_+$ with
$0< \ov{r(t;\beta)} < \bar{\beta} +\bar{c}$ such that $u(t,x)=U(t,R(t;\beta)-x;\beta-r)$
(with $R(t;\beta) := \int_0^t r(s;\beta) ds$) solves the equation \eqref{p}$_1$ for $t\in \R,\ x< R(t)$,
and $r(t;\beta)=-\mu(t)u_x(t,R(t;\beta))= A[\beta-r]$.

Moreover,
\begin{itemize}
\item[(i)] both $r(t;\beta)$ and $\beta- r(t;\beta)$ are increasing in $\beta$ in the sense that
      $r(t;\beta_1) < r(t;\beta_2)$ and $\beta_1 - r(t;\beta_1)\leqslant,\not\equiv
    \beta_2 - r(t;\beta_2)$ if $\beta_1,\, \beta_2\in \PP$  satisfy $\overline{\beta_1},\, \overline{\beta_2}\geqslant 0$
    and $\beta_1 \leqslant , \not\equiv \beta_2$;

\item[(ii)] let $\theta\in \PP^0$ be a given function and consider $\beta$ with ``shape" $\theta$, that is,
consider $\beta := b +\theta $ for $b\geqslant 0$. Then $\min_{t\in [0,T]} r(t; b +\theta ) \to \infty$ and
$b- \overline{r(t;b +\theta )} \to \infty$ as $b\to \infty$.
\end{itemize}
\end{prop}

\begin{proof}
By Lemma \ref{lem:p tsw k half}, for any $r\in \PP$, the problem \eqref{p Dinfty} with $k= \beta -r$ has a maximal
(bounded and nonnegative) solution $U(t,z; \beta -r)$, and $A[\beta-r](t)=\mu(t)U_z(t,0; \beta-r)$
is non-increasing in $r$.

Since $\bar{\beta}\geqslant 0$, when $r=r_*:=0$ we have $A [\beta-r_*] = A[\beta]
= \mu(t) U_z (t,0; \beta)> 0= r_* $. When $r= r^* := \bar{\beta} +\bar{c}+A [\beta]$ we have
$\bar{\beta} - \ov{r^*} = -\bar{c}-\overline{A [\beta]} < -\bar{c}$.
It follows from Lemma \ref{lem:p tsw k half} that $A [\beta-r^*]= \mu(t) U_z (t,0; \beta-r^* ) =0 <r^*$.
Set $S:=\{r\in \PP  : r_* \leqslant  r \leqslant  r^*\}$, then as in the proof of \cite[Theorem 2.4]{DGP}
one can show that the mapping $A [\beta -\cdot]$ maps $S$ continuously into a precompact set in $S$.
Using the Schauder fixed point theorem we see that there exists $r(t;\beta)\in S$ such that
$r(t;\beta) = A [\beta(t) -r(t;\beta)]$. Clearly, $r(t;\beta)\geqslant, \not\equiv 0$ and so
$A[\beta(t) -r (t;\beta)] = \mu(t) U_z(t,0;\beta -r)\geqslant, \not\equiv 0$.
This implies by Lemma \ref{lem:p tsw k half} that $\bar{\beta} -\bar{r}> -\bar{c}$ and $U_z(t,0;\beta -r) >0$
for all $t\in [0,T]$. This yields $r(t;\beta)\in \PP_+$.
The uniqueness of $r$ follows from Remark \ref{rem:unique r}. Finally, a
direct calculation shows that the function $u=U(t,R(t;\beta)-x; \beta-r)$ with $R(t;\beta):= \int_0^t r(s;\beta) ds$
solves the equation \eqref{p}$_1$ in $\R \times (-\infty, R(t;\beta))$.

(i). Assume $\beta_1,\, \beta_2 \in \PP$ satisfy $\overline{\beta_1}, \, \overline{\beta_2}\geqslant 0$ and
$\beta_1 \leqslant, \not\equiv \beta_2$. Denote $r_i := r(t;\beta_i)\ (i=1,2)$ for convenience.
Then $r_1 = A[\beta_1 -r_1] <  A[\beta_2 -r_1]$. The strict inequality follows from Lemma \ref{lem:p tsw k half}.
This, together with Remark \ref{rem:unique r}, implies that the unique fixed point $r_2$ of the mapping $A[\beta_2 -\cdot]$
satisfies $r_1 \leqslant, \not\equiv r_2$.

Similarly, since for $r_3:=\beta_2-\beta_1+r_1$ we have $A[\beta_2- r_3] = A[\beta_1-r_1]=r_1\leqslant, \not\equiv r_3$,
by Lemma \ref{lem subso} and Remark \ref{rem:unique r} we have $r_2 \leqslant, \not\equiv r_3$,
which implies that $\beta_1 - r_1 \leqslant,\not\equiv \beta_2 - r_2$.
Using Lemma  \ref{lem:p tsw k half} again we have $A[\beta_1 - r_1] < A[\beta_2 - r_2]$, that is,
$r_1 <r_2$.

(ii). For any given $\theta \in \PP^0$ we consider $\beta$ with the form $b+\theta$.
Under the assumption $\HH$, we can construct a Fisher-KPP type of nonlinearity $f_0(u)$ such that $f_0(u) \leqslant  f(t,u)$ ($t\in [0,T]$,
$u\geqslant 0$), $f'_0(0) =a_0 := \min_{t\in [0,T]} a(t) >0$ and $f_0(0)=f_0 (s_0)=0$ for some
$s_0 \in (0,P_0)$ with $P_0 := \min_{t\in [0,T]} P(t)$. Take $b_0>0$ large and $0<\delta<1/2$ small such that
\begin{equation}\label{b delta}
\mu(t)(1-2\delta)s_0>\delta,\quad \delta b_0+\theta(t)>0
\mbox{\ \  for } t\in [0,T].
\end{equation}
From now on we consider $b$ satisfying $b>b_0$. Consider the problem
$$
q_{zz} + (1-2\delta)b \; q_z + f_0(q) =0\ (z>0),\quad q(0)=0,
\quad q(\infty)=s_0 \mbox{ and } q_z(z)>0\ (z>0).
$$
By the phase plane analysis as in \cite{GLZ} we see that this problem has a unique solution $q$. Denote $\hat{q}:= q_z$,
then
$$
\frac{d \hat{q}}{dq} = (2\delta -1)b -\frac{f_0(q)}{\hat{q}} < (2\delta -1) b, \quad q\in [0,s_0).
$$
Integrating this inequality over $q\in [0,s_0)$ we have
\begin{equation}\label{q_z(0)>}
q_z(0)=\hat{q}(0) >(1-2\delta) b\, s_0.
\end{equation}

For sufficiently large $\ell>0$, consider the problem
$$
\left\{
 \begin{array}{ll}
 v_t=v_{zz}+(b+\theta-\delta b)v_z +f(t,v),
 & z\in (0,\ell),\ t>0,\\
 v(t,0)=0,\ v(t,\ell) = P^0, & t>0,\\
 v(0,z)=P^0 \cdot \chi_{[0,\ell]} (z), & z\in [0,\ell].
 \end{array}
 \right.
$$
Thanks to \eqref{b delta}, we have $b+\theta-\delta b > (1-2\delta)b$. It follows from the comparison
principle that $v(t+ nT,z)\geqslant q(z)$ for all $z\in [0,\ell], t>0$ and integer $n>0$.
Taking limit as $n\to \infty$ we have
$$
U_1(t,z; b+\theta -\delta b,  \ell) \geqslant q(z), \quad z\in [0,\ell],\ t\in [0,T].
$$
Taking limit again as $\ell \to \infty$ we obtain
$$
U(t,z; b+\theta -\delta b ) \geqslant q(z),\quad z\geqslant 0,\ t\in [0,T].
$$
This, together with \eqref{b delta} and \eqref{q_z(0)>}, implies
that
$$
A[b+\theta -\delta b] =\mu(t) U_z(t,0; b+\theta -\delta b) \geqslant \mu(t) q_z(0) > \mu(t) (1-2\delta ) b\, s_0 >\delta b.
$$
In other words, $r_4 := \delta b$ satisfies $r_4 < A[b+\theta -r_4]$. Since $r_5 := r(t; b+\theta)$ satisfies
$r_5 = A[b +\theta -r_5]$, by Lemma \ref{lem subso} we have
\begin{equation}\label{rto1}
r(t; b+\theta) = r_5 \geqslant r_4 = \delta b\to\infty\ \ \mbox{ as }\ b\to\infty.
\end{equation}

Finally, let us employ an indirect argument to prove
\begin{equation}\label{ybto1}
y (b) := b-\overline{r(t; b+\theta)}\to \infty\ \
\mbox{ as }\ b\to \infty.
\end{equation}
When $b_2 > b_1 \geqslant 0$, it follows from (i) that $b_1 +\theta - r(t;b_1 +\theta) \leqslant,\not\equiv
b_2 +\theta - r(t;b_2 +\theta)$. Taking average over $t\in [0,T]$ we have
$b_1 - \ov{r(t;b_1 +\theta)} < b_2  - \ov{r(t;b_2 +\theta)}$. Hence $y(b)$ is a strictly
increasing function. Assume that, for some constant $C_1 > 0$, $y(b)\leqslant  C_1 $ for all $b>0$, we are
going to derive a contradiction. Set
$$
\begin{array}{c}
\theta_1(t):= \theta(t) - r(t;b+\theta) + \ov{r(t;b+\theta)}\in \PP^0,\quad \Theta_1(t):= \int_0^t \theta_1(s) ds,\\
\mbox{and \ \ } v_1(t,z):= U(t,z-\Theta_1(t); b+\theta - r(t;b+\theta)).
\end{array}
$$
Then $v_1$ satisfies
$$
\left\{
 \begin{array}{ll}
 v_{1t}=v_{1zz}+\big(b-\ov{r(t;b+\theta)}\big)v_{1z}+ f(t,v_1),
 & z> \Theta_1(t),\ t\in [0,T],\\
 v_1 (t, \Theta_1(t))=0, & t>0,\\
 v_1 (0,z)= v_1(T,z),\ \ v_{1z}(t,z)>0,
 & z\geqslant \Theta_1(t),\ t\in [0,T].
 \end{array}
 \right.
$$
Assume $z_1:= \min_{t\in [0,T]} \Theta_1(t)$ is attained at $t=t_1\in [0,T)$. Since $\Theta_1(t)$
is a $C^1$ function we have $\Theta'_1 (t_1) = \theta_1(t_1)= 0$. Combining with
$\min_{t\in [0,T]} r(t;b+\theta)\to \infty\ (b\to \infty)$ we have
\begin{equation}\label{rrto1}
r(t_1;b+\theta) =\theta(t_1)+\ov{r(t;b+\theta)}\to\infty
\ \ \ \mbox{ as }\ b\to\infty.
\end{equation}
For large $\ell$, we compare $v_1$ with the solution $v$ of the following problem
$$
\left\{
 \begin{array}{ll}
 v_t=v_{zz}+ C_1 v_z + f(t,v), & z\in (z_1,\ell + z_1),\ t>0,\\
 v(t,z_1)=0,\ v(t,\ell +z_1 ) = P^0 , & t>0,\\
 v(0,z)=P^0 \cdot \chi_{[z_1,\ell+z_1]} (z), & z\in [z_1,\ell +z_1].
 \end{array}
 \right.
$$
Using the comparison principle we have
$$
v_1(t+nT,z)\leqslant  v(t+nT,z)\mbox{ for } \Theta_1(t)
\leqslant  z\leqslant  \ell +z_1,\ t\in [0,T] \mbox{ and  any integer } n \geqslant 0.
$$
Taking limit as $n\to \infty$ we have
$$
v_1(t,z)\leqslant  U_1 (t,z-z_1;C_1 ,\ell),\quad \Theta_1(t) \leqslant  z\leqslant  \ell +z_1,\ t\in [0,T].
$$
In particular, at $t= t_1$ we have $z_1=\Theta_1 (t_1)$ and
$$
U_{1z} (t_1, 0; C_1 , \ell) \geqslant v_{1z}(t_1, z_1) = U_z(t_1, 0; b+\theta - r(t;b+\theta)),
$$
which implies that
$$r(t_1; b+\theta)=\mu(t_1)U_z(t_1,0; b+\theta- r(t;b+\theta)) \leqslant  \mu(t_1) U_{1z}(t_1, 0; C_1,\ell)<\infty,
$$
contradicting \eqref{rrto1}. Thus \eqref{ybto1} holds and the proof is complete.
\end{proof}

\begin{remark}\rm
This proposition is a key result in the following argument. We remark that the assumption $a(t)>0$
in $\HH$ is only used to ensure that $a_0 = \min a(t) >0$ in the above proof for (ii).
All the conclusions which are not related to this proposition (to say, for small advection
problems) can be extended to some cases where $a(t)$ changes sign but $\bar{a}>0$.
\end{remark}

\medskip
\noindent
{\bf (II). Periodic leftward traveling semi-waves}.

\begin{prop}\label{prop:tsw left}
Assume $0\leqslant  \bar{\beta}<\bar{c}$. Then there exists a unique function $l(t;\beta)\in\PP_+$ with
$0<\ov{l(t;\beta)} < \bar{c} -\bar{\beta}$ such that, with $L(t; \beta):=
\int_0^t l(s;\beta)ds$, $u(t,x)= U (t, x+ L(t;\beta); -\beta-l)$ solves the equation \eqref{p}$_1$ for
$t\in \R,\ x>-L(t;\beta)$,
and $l(t;\beta)= \mu(t) u_x (t, L(t;\beta)) = A [-\beta -l]$.
\end{prop}

\begin{proof}
Note that $A[-\beta - \cdot ]$ maps the set $S_l := \{l\in \PP : 0\leqslant l \leqslant -\bar{\beta}+\bar{c} +A[-\beta]\}$
continuously into a precompact set in $S_l$.  The rest proof is similar as that in Proposition \ref{prop:tsw right}.
\end{proof}

\medskip
\noindent
{\bf (III). Periodic traveling wave $Q \big( t, x + \bar{c} t -\int_0^t \beta(s) ds \big)$}.
It is known (cf. \cite{BH1,H1,HR}) that the equation $u_t = u_{xx} + f(t,u)$ has many {\it periodic
traveling waves} of the form $q \big(t, x + c_1 t; c_1 \big)$, where for any $c_1
\geqslant \bar{c}$, $q (t,z; c_1)$ is the solution of
$$
\left\{
 \begin{array}{l}
v_t = v_{zz} - c_1 v_z +f(t,v) \mbox{ and } v_z(t,z)>0
\quad \mbox{ for } t\in [0,T],\ z\in \R, \\
v (t,-\infty)=0,\ v(t,\infty)= P(t)
\mbox{ and } v(0,0)= (1/2)P_0 := (1/2) \min_{0\leqslant  t\leqslant  T} P(t) ,\\
 v (0,z) = v(T,z) \mbox{ for } z\in \R.
 \end{array}
 \right.
$$
Denote $Q(t,z) := q(t,z; \bar{c})$. Then $Q(t,z)$ is $T$-periodic in $t$ and $Q\big(t,x +\bar{c}t\big)$
is the periodic traveling wave of $u_t=u_{xx}+f(t,u)$ with minimal average speed $\bar{c}$.
It is easily seen that the function $u= Q^*(t,x) := Q \big( t, x + \bar{c} t -\int_0^t \beta(s) ds \big)$ solves the following problem
$$
\left\{
 \begin{array}{l}
 u_t = u_{xx} - \beta(t) u_x +f(t,u) \mbox{ and } u_x (t,x)>0
 \quad \mbox{ for } t, x\in \R, \\
 u(t,-\infty)=0,\ u(t,\infty)= P(t) \mbox{ and }
 u(0,0)= P_0 /2,\\
 u (t+T,x) = u(t,x +X) \mbox{ for } t,
 x\in \R \mbox{ and } X:= (\bar{c}-\bar{\beta})T.
 \end{array}
 \right.
$$
So $Q^*$ is a periodic traveling wave of \eqref{p}$_1$. We remark that $Q^*$ coincides with the definition
for periodic traveling waves as in \cite{BH1,H1,HR}. In fact, by setting $\mathcal{Q}(t,z):= Q \big(t,z -
\int_0^t \tilde{\beta} (s) ds \big)$ (which is $T$-periodic in $t$), we see that $Q^*$ can be expressed
as $\mathcal{Q}\big(t, x + (\bar{c} -\bar{\beta} )t \big)$.

\subsection{The set of critical advection functions}
When $\bar{\beta} >\bar{c}$, the periodic traveling wave $Q^*$ moves rightward with average speed
$\bar{\beta} -\bar{c}$, which can be used to characterize the motion of the back of the
solution $u$ (that is, the sharp increasing part of the solution, cf. \cite{GLZ}).
On the other hand, the periodic traveling semi-wave $U(t,R(t;\beta)-x; \beta -r(t;\beta))$
also moves rightward with speed $r(t; \beta)$, which can be used to characterize the propagation of the front
(that is, the sharp decreasing part of the solution, cf. \cite{DuLin, DuLou, GLZ}). A natural question
is: whether $\mathbf{v}(\beta):= \overline{r(t;\beta)} - [\bar{\beta} -\bar{c} ]$ is positive or
negative. For homogeneous problem, it was shown in \cite{GLZ} that $\mathbf{v}(\beta)$ is strictly
decreasing in $\beta$ and it has a unique zero $\beta^* \ (>\bar{c})$.
This is another critical value for $\beta$ (the first one is $\bar{c}$) which partitions the medium-sized and the large
advections. For time-periodic problem \eqref{p} we show that such critical advection functions also exist, but the
situation is much more complicated since it depends on the ``shape" $\tilde{\beta}$.
We will prove that the set of such functions can be expressed by the following equivalent sets.
\begin{equation}\label{def:B B'}
\BBB:= \{ B(\theta) +\theta  : \theta \in \PP^0\},\qquad \BBB' :=\{ A[\bar{c}+\omega ]+\bar{c}+\omega  : \omega  \in \PP^0\}.
\end{equation}

\begin{lem}\label{lem:B(theta)}
For any given $\theta\in \PP^0$, there exists a unique $B(\theta)\in\R$ such that
$B(\theta) >\bar{c}$ and
 \begin{equation}\label{B theta}
\begin{array}{c}
b-\bar{c}< \ov{r(t;b+\theta)} \mbox{ when } 0\leqslant b<B(\theta),\quad
b -\bar{c}=\ov{r(t; b+\theta)} \mbox{ when } b =B(\theta),\\
b-\bar{c} > \ov{r(t;b+\theta)} \mbox{ when } b>B(\theta).
\end{array}
\end{equation}
\end{lem}

\begin{proof}
For $\theta\in \PP^0$, define a function
$y:[\bar{c},\infty)\to\R$ by
\begin{equation}\label{def y}
y(b) := b - \bar{c} - \ov{r(t; b+\theta)}
\mbox{ for } b\geqslant \bar{c}.
\end{equation}
By Proposition \ref{prop:tsw right}, $y(b)$ is a strictly increasing function in $b\in [\bar{c}, \infty)$,
$y(\bar{c})<0$ and $y(b)>0$ for all large $b$.
The proof is complete once the continuity of $y(b)$ is established. In fact, for any $b_*\in[\bar{c},\infty)$,
assume that $b_n$ decreases and converges to $b_*$ as $n\to \infty$. It then follows from the monotonicity of $b-\ov{r(t;b+\theta)}$ with respect to $b$ in Proposition \ref{prop:tsw right} that $b_*-\ov{r(t;b_* +\theta)}<b_n-\ov{r(t;b_n+\theta)}$, which means that
\[
0<\overline{r(t; b_n +\theta)}-\overline{r(t; b_* +\theta)} <b_n -b_*.
 \]
This proves the continuity of $\ov{r(t; b+\theta)}$ on the right side
of $b_*$. When $b_* >\bar{c}$, the continuity of $\ov{r(t; b+\theta)}$
on the left side of $b_*$ is proved similarly. Thus $\ov{r(t; b+\theta)}$ is continuous in $b\in[\bar{c}, \infty)$, so is $y(b)$, as we wanted. The proof is complete.
\end{proof}

\medskip
From this lemma we see that the function $B(\theta)$ and the set $\BBB$ are well defined.
To understand further properties of $\BBB$ we now give another description for it.

\begin{lem}\label{lem:property of beta*}
Assume that $\beta^*:= A[\bar{c} +\omega ]+\bar{c} +\omega $ for some $\omega \in \PP^0$. Then
\begin{equation}\label{beta* 3}
\bar{\beta} - \bar{c} < \overline{r(t;\beta)} \mbox{ when } \beta <\beta^*,\quad
\bar{\beta} - \bar{c} = \overline{r(t;\beta)} \mbox{ when } \beta =\beta^*,\quad
\bar{\beta} - \bar{c} > \overline{r(t;\beta)} \mbox{ when } \beta >\beta^*.
\end{equation}
\end{lem}

\begin{proof}
Denote $r_1 := A[\bar{c} +\omega]$. By the definition of $\beta^*$ we have $\beta^* - r_1 = \bar{c}+\omega$, and so
 $A[\beta^* - r_1] = A[\bar{c}+\omega] =r_1$. Thus $r(t;\beta^*)=r_1$ by Proposition \ref{prop:tsw right}.
Moreover, $\ov{\beta^*}- \ov{r_1} = \ov{\beta^*} - \ov{r(t; \beta^*)} = \bar{c}$. This proves the equality in \eqref{beta* 3}.

Now we consider the case $\beta<\beta^*$. It follows from Proposition \ref{prop:tsw right} that
$\beta - r(t;\beta) \leqslant,\not\equiv \beta^*- r(t;\beta^*)$. Hence
$\bar{\beta}- \overline{r(t;\beta)}< \overline{\beta^*}-\overline{r(t;\beta^*)}$.
This, together with the equality in \eqref{beta* 3}, yields that
$$
\overline{\beta}-\bar{c}<\overline{\beta^*}-\overline{r(t;\beta^*)}
+\overline{r(t;\beta)}-\bar{c}=\overline{r(t;\beta)},
$$
which proves the first inequality in \eqref{beta* 3}. The last inequality is proved similarly.
\end{proof}

\begin{lem}\label{lem:B=B'}
Let $\BBB$ and $\BBB'$ are defined by \eqref{def:B B'}, then $\BBB = \BBB'$.
\end{lem}

\begin{proof}
We first prove $\BBB' \subset\BBB$. For any given $\omega \in\PP^0$, set
$$
r_0 := A[\bar{c}+\omega], \quad \theta:=\omega  +\wt{r_0} , \quad B^*:= \ov{r_0} +\bar{c}
\mbox{\ \ and \ \ } \beta^* := A[\bar{c} +\omega ] + \bar{c} +\omega  =  B^* +\theta.
$$
Then $ r_0$ is a fixed point of $A[\beta^* -\cdot]$, and so $r(t;\beta^*)=r(t; B^* +\theta) = r_0$.
Thus the function $y(b)$ defined by \eqref{def y} satisfies
$$
y(B^*) = B^* -  \bar{c} - \ov{r(t; B^* +\theta)} =0,
$$
that is, $B^*=B(\theta)$. Thus $\beta^*=B(\theta)+\theta\in\BBB$.

Next we show that $\BBB\subset\BBB'$. For any given $\theta\in \PP^0$, denote $\beta_*
:= B(\theta) +\theta$, $r_1 (t):= r(t; \beta_* )$ and $\omega := \theta - \wt{r_1}$. It follows from Lemma \ref{lem:B(theta)} that $B(\theta)-\bar{c} =\ov{r_1}$, then
$$
\beta_*= B(\theta)+ \theta =\bar{c} +\ov{r_1} +\theta=\bar{c}+
\ov{r_1} +\omega +\wt{r_1}=\bar{c} +\omega + r_1,
$$
which implies that $\beta_*-r_1=\bar{c} +\omega $ and thus
$r_1= r(t;\beta_*)= A[\beta_* -r_1 ]= A[\bar{c} +\omega ]$.
Consequently we have $\beta_*=\bar{c} +\omega + r_1 =\bar{c} +\omega  + A[ \bar{c} +\omega  ]\in \BBB'$.
\end{proof}

\begin{remark}\rm
Lemmas \ref{lem:B(theta)} and \ref{lem:property of beta*} explain the constructions of
$\BBB$ and $\BBB'$, respectively. The equivalence in Lemma \ref{lem:B=B'} implies that,
for any $\theta\in \PP^0$, there exists $\omega\in \PP^0$ such that $B(\theta)+\theta =
A[\bar{c}+\omega] +\bar{c} +\omega$, and vice versa.
The notation in $\BBB$ is convenient when we regard $B(\theta)+\theta$ as an analogue of
the second critical value $\beta^*$ for homogeneous problems. Namely, for any
given ``shape" $\theta$ we have a critical value $B(\theta)$ which corresponds to the second
critical function $B(\theta)+\theta$. The notation in $\BBB'$, however, is convenient to
estimate the spreading speed $\ov{r(t;\beta^*)}$ of the periodic rightward traveling semi-wave
(which equals to the speed $\ov{\beta^*} -\bar{c}$ of the periodic traveling wave $Q^*$ when $\beta=\beta^*$).
It turns out that this speed is nothing but $\ov{A[\bar{c}+\omega]}$ for some $\omega\in \PP^0$.
\end{remark}

It is natural to ask whether $\BBB$ or $\BBB'$ can be defined as an equivalence class with the same
average, like $\{\beta  : \bar{\beta}=\ov{A[\bar{c}]}+ \bar{c}\}$ or $\{ \beta  : \bar{\beta} = B(0)\}$.
At the end of this subsection we show that the answer is generally negative.
The main reason is that $A[\bar{c}+\omega ]=\mu(t)U_z(t,0; \bar{c}+\omega )$ depends not only on $\omega $, but
also on $\mu(t)$.

\begin{lem}\label{lem:BBB}
For any given $\omega \in \PP^0$, there exists some $\mu\in \PP_+$ such that $\ov{\beta^*(\omega)} \not=
\ov{\beta^*(0)}$, where $\beta^*(\omega) := A[\bar{c}+\omega ] +\bar{c}+\omega $
and $\beta^*(0) :=  A[\bar{c}] +\bar{c}$.
\end{lem}

\begin{proof}
Denote $\xi_1(t):= U_z(t,0;\bar{c})$ and $\xi_2(t):=U_z(t,0; \bar{c} +\omega )$, then
$A[\bar{c}] = \mu(t) \xi_1(t)$ and $A[\bar{c}+\omega] = \mu(t) \xi_2(t)$. Clearly, to prove
the lemma it suffices to show that $\ov{\mu \xi_1} \not= \ov{\mu \xi_2}$ for some $\mu$.
For clarity we divide the proof into three steps.

\smallskip

{\bf Step 1}. We claim that $\xi_1(t)\not \equiv \xi_2(t)$.

Suppose by way of contradiction that $\xi_1(t) \equiv \xi_2(t)$. Denote $\Theta(t): =\int_0^t \omega (s)ds$, then the
function $\hat{v}(t,z):= U(t,z-\Theta(t); \bar{c}+\omega )$ satisfies $\hat{v}_z (t,\Theta(t)) = \xi_2(t) \equiv \xi_1(t)$ and
$$
\left\{
 \begin{array}{ll}
 \hat{v}_t = \hat{v}_{zz} +\bar{c} \hat{v}_z + f(t,\hat{v}),
 & z> \Theta(t),\ t\in [0,T],\\
 \hat{v}(t, \Theta(t))=0, & t \in [0,T],\\
 \hat{v}(0,z)= \hat{v}(T,z), & z\geqslant \Theta(t),\ t\in [0,T].
 \end{array}
 \right.
$$
Set $z_0 := \min_{t\in [0,T]} \Theta(t)$. For large $\ell$ we compare $\hat{v}$ with the solution $v$ of
$$
\left\{
 \begin{array}{ll}
 v_t = v_{zz} + \bar{c} v_z + f(t,v), & z\in (z_0 , \ell + z_0),\ t>0,\\
 v(t,z_0)=0,\ v(t,\ell +z_0 ) = P^0, & t>0,\\
 v(0,z)=P^0 \cdot \chi_{[z_0,\ell+z_0]} (z), & z\in [z_0,\ell +z_0].
 \end{array}
 \right.
$$
It follows from the comparison principle that
$$
v(t+nT, z) \geqslant \hat{v}(t+nT, z) ,\quad \Theta(t) \leqslant  z\leqslant  \ell +z_0,\ t\in [0,2T] \mbox{ and any integer }
n\geqslant 0.
$$
Taking limit as $n\to \infty$ we have
$$
U_1 (t, z-z_0; \bar{c}, \ell) \geqslant \hat{v}(t,z), \quad  \Theta(t) \leqslant  z\leqslant  \ell +z_0,\ t\in [0,2T].
$$
Taking limit again as $\ell \to \infty$ we have
$$
U (t, z-z_0; \bar{c}) \geqslant \hat{v}(t,z),
\quad  z\geqslant \Theta(t),\ t\in [0,2T].
$$
Choose $t_1 \in [0,T)$ such that $\Theta(t_1) > z_0$. Then
the strong comparison principle yields that
$$
U (t, z-z_0; \bar{c}) > \hat{v}(t,z),
\quad  z\geqslant \Theta(t),\ t\in [t_1, t_2),
$$
where $t_2:= \max\{\tau \in (t_1,t_1 +T)  : \Theta (t)>z_0 \mbox{ for } t\in (t_1, \tau)\}$.
At $t=t_2$, by the comparison principle and the Hopf boundary lemma we have
$$
U_z(t_2, 0; \bar{c}) > \hat{v}_z (t_2, z_0)= \xi_2(t_2) = \xi_1(t_2) = U_z(t_2, 0;\bar{c}).
$$
This contradiction proves $\xi_1(t)\not \equiv \xi_2(t)$.

\smallskip

{\bf Step 2}. We shall show that when $\overline{\xi_1}= \overline{\xi_2}$, there exist some
$\mu\in \PP_+$ such that $\overline{\mu \xi_1} \not= \overline{\mu \xi_2}$.

In fact, it follows from Step 1 that there is a closed interval $J\in [0,T]$ such that $\xi_1(t)<\xi_2(t)$.
Hence $\xi_1(t)+\delta<\xi_2(t)$ for some small $\delta>0$. Define a function $\mu_0$ by
$$
\mu_0(t) :=
 \left\{
  \begin{array}{ll}
  2, &  t\in J,\\
  1, &  t\in  [0,T]\backslash J.
  \end{array}
  \right.
$$
Then
$$
\int_0^T \mu_0(t) [\xi_2(t) -\xi_1(t)] dt
= \int_J [\xi_2(t) -\xi_1(t)] dt + \int_0^T [\xi_2(t) -\xi_1(t)]dt > \delta |J|.
$$
Let $\mu$ be a function in $\PP_+$ obtained by smoothenning $\mu_0(t)$ slightly, then
$$
 \overline{\mu \xi_2 } -\overline{\mu \xi_1 } = \frac{1}{T} \int_0^T \mu(t) [\xi_2(t) -\xi_1(t)] dt >0.
$$
This completes the proof of Step 2.

\smallskip

{\bf Step 3}. Finally we consider the case where $\overline{\xi_1} \not= \overline{\xi_2}$.
In this case it is clear that $\overline{\mu \xi_1 } \not= \overline{\mu \xi_2} $ for any
constant $\mu$.

The proof of the lemma is complete.
\end{proof}

\subsection{Traveling waves with compact supports}
In this subsection we present some periodic traveling waves with compact supports, which
will be used to support the solution of \eqref{p} from below so that spreading happens.
This subsection can be regarded as a supplement to Subsection 3.2.

\begin{prop}\label{prop:tw compact large}
Assume that $\beta\in \PP$ and $\bar{\beta}\geqslant 0$.
\begin{itemize}
\item[(i)] For any small $\delta>0$, when $\ell>0$ is sufficiently large, the function
$u= W(t,x):= U_0(t, R_1(t)-x; \beta-r_1, \ell)$ (with $r_1 := \bar{\beta}-\bar{c}+\delta $
and $R_1 (t) := \int_0^t r_1(s) ds$) solves \eqref{p}$_1$ in $\R \times (R_1(t)-\ell, R_1(t))$.

\item[(ii)] Assume further that $0\leqslant \bar{\beta} < B(\tilde{\beta})$. Then
for any small $\epsilon >0$, when $\ell>0$ is large enough, the function
$u= W^\epsilon (t,x) := U_0(t, R^\epsilon (t)-x; \beta-r^\epsilon, \ell)$
(with $r^\epsilon := r(t;\beta)-\epsilon$ and $R^\epsilon (t) := \int_0^t r^\epsilon (s) ds$)
is a lower solution of \eqref{p} in the sense that, it solves \eqref{p}$_1$
in $\R \times (R^\epsilon (t)-\ell, R^\epsilon(t))$, and $r^\epsilon < -\mu(t)
W^\epsilon_x (t,R^\epsilon(t)) = A_0[\beta-r^\epsilon, \ell]$.

\item[(iii)] Assume that $0\leqslant \bar{\beta} < \bar{c}$. Then for any small $\epsilon >0$, when
$\ell>0$ is sufficiently large, the function $u= W^\epsilon_l (t,x) := U_0(t, x+L^\epsilon (t);
-\beta-l^\epsilon, \ell)$ (with $l^\epsilon := l(t;\beta)-\epsilon$ for
$l$ given in Proposition \ref{prop:tsw left} and $L^\epsilon (t) := \int_0^t l^\epsilon (s;\beta) ds$)
is a lower solution of \eqref{p} in the sense that, it solves \eqref{p}$_1$
in $\R \times (-L^\epsilon (t), \ell - L^\epsilon(t))$, and $l^\epsilon < \mu(t)
W^\epsilon_{lx} (t, -L^\epsilon(t)) = A_0[-\beta-l^\epsilon, \ell]$.
\end{itemize}
\end{prop}

\begin{proof}
(i). For $r_1 := \bar{\beta}-\bar{c}+\delta$ we have $\bar{\beta}- \ov{r_1} = \bar{c} -\delta \in (-\bar{c}, \bar{c})$, provided $\delta>0$
is small. By Lemma \ref{lem:positive sol DD0}, the problem \eqref{p DD0}
with $k=\beta -r_1$ has a unique positive solution $U_0(t,z; \beta-r_1,
\ell)$ for each large $\ell>0$. A direct calculation shows that
$W(t,x):= U_0(t, R_1(t)-x; \beta-r_1, \ell)$ (with $R_1 (t) := \int_0^t r_1(s) ds$) solves \eqref{p}$_1$ in $\R\times (R_1(t)-\ell, R_1(t))$.
This (compactly supported) traveling wave travels rightward/leftward if $\ov{r_1} >0$/$\ov{r_1}<0$.

(ii). By Proposition \ref{prop:tsw right}, there exists $r(t;\beta)$ which satisfies $0<\bar{r}< \bar{\beta}+\bar{c}$ and $r=A[\beta -r]$.
On the other hand, by our assumption $0\leqslant \bar{\beta} <B(\tilde{\beta})$ and by Lemma \ref{lem:property of beta*}, we have $\bar{\beta}-\bar{c}<\bar{r}$. Therefore $| \bar{\beta}-\bar{r}| < \bar{c}$. Consequently, there exists $\epsilon_0 >0$
small such that for any $r^\epsilon := r -\epsilon $ (with $\epsilon \in (0,\epsilon_0)$) we have $| \bar{\beta}-\ov{r^\epsilon}| < \bar{c}$.
It follows from Lemma \ref{lem:positive sol DD0} that the problem
\eqref{p DD0} with $k=\beta -r^\epsilon$ has a unique positive solution $U_0(t,z; \beta-r^\epsilon, \ell)$ for each large $\ell>0$. A direct
calculation shows that $W^\epsilon(t,x):= U_0(t, R^\epsilon (t)-x; \beta-r^\epsilon, \ell)$ (with $R^\epsilon (t):=
\int_0^t r^\epsilon (s) ds$) solves \eqref{p}$_1$ in $\R \times (R^\epsilon (t)-\ell, R^\epsilon (t))$.

Moreover, $r^\epsilon = r-\epsilon  < r =A[\beta -r] < A[\beta -r^\epsilon]$ by
Lemma \ref{lem:p tsw k half}. This, together with Lemma
\ref{lem:A0 ell to A}, implies that $r^\epsilon <A_0 [\beta -r^\epsilon, \ell] =
-\mu(t) W^\epsilon_x (t, R^\epsilon (t))$ when $\ell$ is sufficiently large.
This means that $W^\epsilon$ is a lower solution of the problem \eqref{p}.

(iii). The proof is similar as (ii) by using Proposition \ref{prop:tsw left} instead of Proposition
\ref{prop:tsw right}.
\end{proof}

\section{Long time behavior of the solutions}\label{seclo}
In this section we study the influence of $\beta(t)$ on the asymptotic behavior of the solutions. In Subsection 4.1
we focus on the small advection case $0\leqslant\bar{\beta} <\bar{c}$ and prove Theorem \ref{thm:small beta}. In
Subsection 4.2 we study the boundedness of $g_\infty $ and $h_\infty$ when $\bar{\beta} \geqslant \bar{c}$.
Theorem \ref{thm:large beta} then follows easily. In Subsection 4.3 we deal with the medium-sized advection
case (i.e., $\bar{c}\leqslant\bar{\beta} <B(\tilde{\beta})$) and prove Theorem \ref{thm:middle beta}.

\subsection{Small advection case}\label{sac}
We start with the following equivalent conditions for vanishing.

\begin{lem}\label{lemvansmall}
Assume $0\leqslant \bar{\beta}<\bar{c}$.  Then the following three assertions are equivalent:
$$
{\rm (i)}\  h_\infty \mbox{ or } g_\infty \mbox{ is finite};\qquad
{\rm (ii)}\  h_\infty-g_\infty\leqslant  \ell^* (-\beta,a); \qquad  {\rm (iii)}\ \|u(t,\cdot)\|_{L^\infty ([g(t),h(t)])} \to 0
\mbox{ as } t\to \infty.
$$
\end{lem}

\begin{proof}
``(i)$\Rightarrow$ (ii)". Without loss of generality we assume $g_\infty > -\infty$ and
prove (ii) by contradiction. Assume that $h_\infty-g_\infty>\ell^* (-\beta , a)$, then for sufficiently large
integer $n_1$ and $t_1 := n_1 T$, we have $h(t_1) - g(t_1) >  \ell^*$.

Now we consider an auxiliary problem:
\begin{equation}\label{subso}
\left\{
\begin{array}{ll}
v_t = v_{xx} - \beta(t) v_x +f(t, v), & t> t_1,\ x\in (\xi(t), h(t_1)),\\
v(t, \xi(t)) = 0,\quad \xi'(t)= -\mu(t) v_x(t, \xi(t)),& t>t_1,\\
v (t,h(t_1))=0, & t> t_1,\\
\xi(t_1) = g(t_1),\ \  v(t_1, x)= u(t_1, x), & x\in [g(t_1), h(t_1)].
  \end{array}
 \right.
 \end{equation}
Clearly, $v$ is a lower solution of \eqref{p}. So $\xi(t)\geqslant g(t)$ and $\xi(\infty)>-\infty$
by our assumption.
Using a similar argument as in \cite[Lemma 3.3]{DGP} by straightening the free boundary
one can show that
$$
\|v(t,\cdot)- V(t,\cdot)\|_{C^2([\xi(t), h(t_1)])} \to 0,\quad \mbox{as } t\to\infty,
$$
where $V(t,x):= U_0(t,x-\xi(\infty);-\beta, \ell)$
and $U_0(t,z;-\beta,\ell)$ is the periodic solution of \eqref{p DD0} with $k=-\beta$ and
$\ell := h(t_1) - \xi(\infty)> h(t_1) -\xi(t_1) >\ell^*$. Therefore,
\[
\liminf_{t\to\infty} \xi'(t) = \liminf_{t\to\infty} [ -\mu (t) v_x(t,\xi(t))] =
\liminf_{t\to\infty} [-\mu(t) V_x (t,\xi(\infty)) ] = -\delta
\]
for some $\delta>0$. This contradicts the assumption $\xi(\infty) >-\infty$.

\smallskip
``(ii)$\Rightarrow$(i)". When (ii) holds, (i) is obvious.

\smallskip

``(ii)$\Rightarrow$(iii)".  By the assumption and \cite[Theorem 28.1]{P} we see that the unique positive
solution of the following problem
\begin{equation}\label{upbsoper}
\left\{
\begin{array}{ll}
v_t=v_{xx}-\beta(t)v_x+f(t,v), & t>0,\ x\in[g_\infty,h_\infty],\\
 v(t,g_\infty)= v(t,h_\infty)=0, &  t>0,\\
 v(0,x)=v_0(x)\geqslant 0, & x\in[g_\infty,h_\infty],
  \end{array}
 \right.
 \end{equation}
with $v_0(x)\geqslant u_0(x)$ for $x\in[-h_0,h_0]$, satisfies $v\to0$ uniformly for $x\in[g_\infty,h_\infty]$ as $t\to\infty$.
Then the conclusion (iii) follows easily from the comparison principle.

\smallskip

``(iii)$\Rightarrow$(ii)": We proceed by a contraction argument. Assume that, for
some small $\varepsilon>0$ there exists a large integer $n_2$ such that
$h(t)-g(t)>\ell^*+ 3\varepsilon$ for all $t>t_2 := n_2 T$.
It is known that eigenvalue problem \eqref{eigen-p} with $\ell = \ell^*+\varepsilon$
and  $k(t)=-\beta(t)$, admits a negative principal eigenvalue, denoted by $\lambda_\varepsilon$, whose corresponding
positive eigenfunction, denoted by $\varphi_\varepsilon$, can be chosen positive and normalized by $\|\varphi_\varepsilon\|_{L^{\infty}}=1$. Set
\[
w(t,x) :=\delta\varphi_\varepsilon(t,x)\ \mbox{ for } (t,x)\in[0,T]\times[0,\ell^*+\varepsilon],
\]
with $\delta>0$ small such that
\[
f(t,\delta\varphi_\varepsilon)\geqslant a(t)\delta\varphi_\varepsilon+
\frac{1}{2}\lambda_\varepsilon\delta\varphi_\varepsilon\ \mbox{in } [0,T]\times[0, \ell^*+\varepsilon].
\]
A simple calculation yields that for $(t,x)\in[0,\infty)\times[0, \ell^*+\varepsilon]$,
$$
w_t-w_{xx}+\beta(t)w_x-f(t,w)  = \delta\varphi_\varepsilon [a(t)+\lambda_\varepsilon ]
-f(t,\delta\varphi_\varepsilon) \leqslant \frac{1}{2}\lambda_\varepsilon\delta\varphi_\varepsilon \leqslant 0.
$$
Furthermore, one can see that
\[
0\leqslant  w(0,x) = \delta \varphi_\varepsilon (0,x) <  u(t_2, x +g(t_2)+\varepsilon),\quad
 x\in [0, \ell^*+\varepsilon],
\]
when $\delta$ is sufficiently small, since the last function is positive on $[0,\ell^* +\varepsilon]$.
By comparison we have
$$
u(t+t_2,x +g(t_2) +\varepsilon) \geqslant w(t,x),\quad (t,x)\in[0,\infty)\times[0, \ell^*+\varepsilon],
$$
contradicting (iii).

This proves the lemma.
\end{proof}

Next, we give a sufficient condition for vanishing, which indicates that if both of
the initial domain and the initial function are small, then the species dies out eventually in
the environment with small advection.

\begin{lem}\label{vfsma}
Assume $0\leqslant \bar{\beta}<\bar{c}$ and let $(u,g,h)$ be a
solution of \eqref{p}. If $h_0<\ell^*(-\beta, a)/2$ and if $\|u_0\|_{L^\infty([-h_0,h_0])}$ is sufficiently small, then
vanishing happens.
\end{lem}

\begin{proof}
We use a similar idea as in \cite{DGP}, but our approach is more complicated
since we are not clear about the symmetry and the monotonicity for the principal eigenfunction.

\smallskip

For any given $h_1\in(h_0,\ell^*(-\beta,a)/2)$, we consider the problem \eqref{eigen-p} with $\ell =2h_1$
and $k(t) = -\beta(t)$. Denote by $\lambda_1$ and $\varphi_1$ with $\|\varphi_1\|_{L^\infty}=1$
the principal eigenvalue and the corresponding positive eigenfunction of this problem.
Then $\lambda_1>0$ by Lemma \ref{lem:1eigenvalue}. We use $\zeta_1(t)$ and $\zeta_2(t)$ to denote the
leftmost and  the rightmost local maximum point of $\varphi_1(t,\cdot)$, respectively.
Set $\eta_1:=\min_{t\in [0,T]}\zeta_1(t)$, $\eta_2:=\max_{t\in [0,T]}\zeta_2(t)$,
\[
 \varepsilon_0 :=\min\Big\{\min_{t\in [0,T]}\varphi_1(t,\eta_1),\ \min_{t\in [0,T]}\varphi_1(t,\eta_2)\Big\}, \
\delta: =\min\Big\{\lambda_1 ,\ \frac{h_1}{h_0}-1 \Big\},
\]
then $\varepsilon_0<1$, and there exists $\varepsilon_1=\varepsilon_1(\delta)>0$
small such that
\[
2\varepsilon_1 \cdot \max \left\{ \max_{t\in [0,T]} |\mu(t) \varphi_{1x}(t,2h_1)|,\ \max_{t\in [0,T]}
[\mu(t) \varphi_{1x}(t,0)] \right\} <\delta^2h_0.
\]
Define
\[
w(t,x) :=\varepsilon_0\varepsilon_1e^{-\delta t}\varphi_1(t,x+h_1)
\ \mbox{ for } (t,x)\in[0,\infty)\times[-h_1,h_1].
\]
A direct calculation shows that for $(t,x)\in[0,\infty)\times[-h_1,h_1]$,
\[
w_t-w_{xx}+\beta(t)w_x-f(t,w)\geqslant
\varepsilon_0\varepsilon_1e^{-\delta t}(\lambda_1 - \delta)\varphi_1(t,x+h_1) \geqslant0.
\]
If we choose $u_0$ satisfying
\[
u_0(x) \leqslant \varepsilon_0 \varepsilon_1 \varphi_1(0,x+h_1) =  w(0,x) \mbox{ for } x\in [-h_0,h_0],
\]
then the comparison principle yields
\begin{equation}\label{uw10}
u(t,x)\leqslant w(t,x)\ \mbox{ for }(t,x)\in[0,\tau )\times[g(t),h(t)],
\end{equation}
where $\tau:=\sup\{t>0  : h(t)<h_1 \mbox{ and }  g(t)>-h_1\}$.
We will prove that $\tau =\infty$. Once this is proved
we have $[g(t), h(t)]\subset [-h_1, h_1]$ for all $t\geqslant0$, and
so the vanishing conclusion follows from the previous lemma.

To prove $\tau =\infty$, we employ an indirect argument by assuming that $\tau<\infty$.
Without loss of generality we may assume that $h(\tau)=h_1$.
Define
\[
\xi(t) :=h_0\Big(1+\delta-\frac{\delta}{2}e^{-\delta t}\Big),\
v(t,x) :=\varepsilon_1e^{-\delta t}\varphi_1(t,x-\xi(t)+2h_1)
\]
for $t\geqslant 0,\ x\in J_r(t):= [\eta_2+\xi(t)-2h_1, \ \xi(t)]$.
A direct calculation shows that
\begin{eqnarray*}
v_t-v_{xx}+\beta(t)v_x-f(t,v)& \geqslant & \varepsilon_1e^{-\delta t}[(\lambda_1- \delta)\varphi_1 (t,x-\xi(t)+2h_1)
- \xi'\varphi_{1x} (t,x-\xi(t)+2h_1)] \\ & \geqslant &  0, \quad t>0,\ x\in J_r(t),
\end{eqnarray*}
since $\xi'>0$ and $\varphi_{1x}(t,x-\xi(t)+2h_1)<0$ for $t\geqslant 0$ and
$x\in J_r(t)$. On the other hand, by the choice of $\varepsilon_1$
we have
$$
\xi'(t)=\frac{\delta^2 h_0}{2} e^{-\delta t} \geqslant
-\mu(t)\varepsilon_1 e^{-\delta t} \varphi_{1x}(t,2h_1)
=- \mu(t) v_x(t, \xi(t)).
$$
We claim that $h(t)\leqslant \xi(t)$ for all $t\in [0,\tau]$.  When $h(t)\leqslant \eta_2+ \xi(t)-2h_1$
the claim is true since $\eta_2+ \xi(t)-2h_1 <\xi(t)$.
Assume that the set $\{ 0\leqslant t\leqslant \tau  : h(t)>\eta_2+ \xi(t)-2h_1\}\not= \emptyset$ consists of
some intervals and $[\tau_1, \tau_2]$ is one of them. Then $h(\tau_1 )=\eta_2+ \xi(\tau_1)-2h_1$, and
on the left boundary $x=\eta_2+ \xi(t)-2h_1$ of the domain $\Omega := \{(t,x)  : \eta_2+ \xi(t)-2h_1 \leqslant x \leqslant h(t),\
t\in [\tau_1, \tau_2]\}$ we have
\begin{eqnarray*}
& & u(t, \eta_2+ \xi(t)-2h_1)\leqslant  w(t, \eta_2+ \xi(t)-2h_1)  =  \varepsilon_0 \varepsilon_1
e^{-\delta t} \varphi_1(t, \eta_2+ \xi(t)-h_1) \\
 & \leqslant &  \varepsilon_0 \varepsilon_1 e^{-\delta t} \leqslant \varepsilon_1 e^{-\delta t} \varphi_1(t, \eta_2)
  \equiv  v(t, \eta_2+ \xi(t)-2h_1),\quad t\in [\tau_1, \tau_2].
\end{eqnarray*}
Hence $v$ is an upper solution in $\Omega$ and by comparison we have $u\leqslant v$ in $\Omega$
and $h(t)< \xi(t)$ for $t\in [\tau_1, \tau_2]$.
(Note that in case $\tau_1 =0$ we need an additional condition: $u_0(x)\leqslant v(0,x)$ for
$\eta_2 +\xi(0)-2h_1 \leqslant x \leqslant h_0$. This is true if we choose $u_0$ sufficiently small.)
In summary, we proved the claim and so $h(\tau)\leqslant \xi (\tau)< \xi (\infty)\leqslant h_1$,
contradicting our assumption $h(\tau)=h_1$. This proves $\tau=\infty$, and the proof of the lemma is complete.
\end{proof}

We now present a sufficient condition for spreading.

\begin{lem}\label{lemuto1}
Assume $0\leqslant \bar{\beta}<\bar{c}$. If $h_0\geqslant \ell^*(-\beta, a)/2$, then spreading happens, that is,
 $-g_\infty=h_\infty=\infty$, and
\begin{equation}\label{utoPt}
\lim_{t\to\infty}[u(t,\cdot)-P(t)]=0 \ \mbox{ locally uniformly in } \R,
\end{equation}
where $P$ is the unique positive $T$-periodic solution of $u_t = f(t,u)$.
\end{lem}

\begin{proof}
Since $g'(t)<0<h'(t)$ for $t>0$, we have $h(t)-g(t)>\ell^*$ for any $t>0$.
So the conclusion $-g_\infty = h_\infty =\infty$ follows from Lemma \ref{lemvansmall}.
In what follows we prove \eqref{utoPt}.

First, using a similar argument as in the proof of \cite[Lemma 3.4]{DGP}
one can show that, for positive integer $m$,
\begin{equation}\label{zzt}
\liminf_{m\to\infty} u(t + mT,x) \geqslant P(t)\
\mbox{ locally uniformly for } (t,x)\in[0,T]\times\R.
\end{equation}
On the other hand, let $v$ be the solution of the following Cauchy problem
\begin{equation}\label{ee4}
\left\{
\begin{array}{ll}
 v_t =v_{xx}-\beta(t)v_x+f(t,v), &
 x\in\R,\ t\in(0,\infty),\\
 v(0,x)=  v_0(x), &  x\in\R,
\end{array}
\right.
\end{equation}
where $v_0 (x)= u_0(x)$ for $x\in [-h_0, h_0]$, and $v_0(x)=0$ for $|x|>h_0$. By comparison we have
\begin{equation}\label{supu1}
u(t,x)\leqslant  v(t,x) \ \mbox{ for all } (t,x)\in[0,\infty)\times[g(t),h(t)].
\end{equation}
Since $0\leqslant \bar{\beta}<\bar{c}$, it follows from \cite[Theorem 1.6]{Nad} that
\[
\lim_{t\to\infty}[v(t,x)-P(t)]=0  \ \mbox{ locally uniformly for } x\in\R.
\]
This, together with \eqref{supu1}, yields that
$$
\limsup_{m\to\infty} u(t + mT,x) \leqslant  P(t)\ \mbox{ locally uniformly for } (t,x)\in[0,T]\times\R.
$$
Combining with \eqref{zzt} we have
$$
\lim_{m\to\infty} u(t + mT,x) =  P(t)\ \mbox{ locally uniformly for } (t,x)\in[0,T]\times\R.
$$
Finally, using the periodicity of $P(t)$ one can easily obtain \eqref{utoPt}.
\end{proof}

\begin{remark}\label{rem12}\rm
Consider the equation $u_t = u_{xx}+k(t)u_x +f(t,u)$ between two variable boundaries $\hat{g}(t)$
and $\hat{h}(t)$. If $\hat{g}(t)\to -\infty$ and $\hat{h}(t)\to \infty$ as $t\to \infty$, then
a similar argument as above shows that $u(t,\cdot)-P(t)\to 0$ as $t\to \infty$, provided
$k\in \PP$ satisfies $|\bar{k}| <\bar{c}$.
\end{remark}

\medskip

\noindent
{\it Proof of Theorem \ref{thm:small beta}.}
Based on the previous lemmas, one can prove Theorem \ref{thm:small beta} easily in a similar way as
proving \cite[Theorem 5.2]{DuLou}.  \qed

\subsection{Boundedness of $g_\infty$ and $h_\infty$}\label{bbbgh}
In this subsection we prove that $g_\infty>-\infty$ when $\bar{\beta}\geqslant \bar{c}$,
and $h_\infty<\infty$ when $\bar{\beta}\geqslant  B(\tilde{\beta})$, which will be used in
the proof of Theorem \ref{thm:large beta}.

First we give some estimates for the solution $u$ of \eqref{p}. Denote
\begin{equation}\label{def:M}
M:=1+\|u_0\|_{L^{\infty}([-h_0,h_0])},
\end{equation}
then it follows from the comparison principle that $u(t,x)\leqslant  M$ for $(t,x)\in[0,\infty)\times[g(t),h(t)]$.
Define a function $f_M(t,u) \in C^{\nu/2, 1+\nu/2}([0,T]\times \R)$ such that
\[
f_M(t,u)\left\{\begin{array}{ll}= a(t)u & \mbox{ for }\ (t,u)\in[0,T]\times[0,1],\\
>0 & \mbox{ for }\ (t,u)\in[0,T]\times(1,M),\\
<0 & \mbox{ for }\ (t,u)\in[0,T]\times(M,\infty),
\end{array}
\right.
\]
and that $f_M(t,u)$ is a Fisher-KPP type of nonlinearity, $T$-periodic in $t$,
\[
(f_M)_u(t,M) <0 \mbox{ and } f(t,u)\leqslant  f_M(t,u)\leqslant  a(t)u \ \mbox{for } (t,u)\in[0,T]\times[0,\infty).
\]
It is known that the equation $u_t=u_{xx}+f_M(t,u)$ has many periodic traveling waves. Denote
the traveling wave with minimal average speed $\bar{c}$ by $Q_M (t,x+\bar{c} t)$,
then $Q_M(t,-\infty) =0$, $Q_M(t,\infty) = M$, $(Q_M)_z (t,z)>0$, $Q_M(t,z)$ is $T$-periodic in $t$ and,
without loss of generality, we may assume that it satisfies the normalization condition: $Q_M (0,0) = M/2$.
Using the principal eigenvalue of time-periodic parabolic operators as in \cite[Subsection 1.3]{H1}
or \cite[Subsection 1.4]{HR}, we see that, for some positive constant $C$,
\begin{equation}\label{qq1}
Q_M(t,z)\sim -Cze^{\frac{\bar{c}}{2}z}  \ \mbox{ as } z\to-\infty,\quad \mbox{uniformly in } t\in [0,T].
\end{equation}
Clearly, $Q_M \big(t, x +\bar{c} t - \int_0^t \beta(s) ds \big)$ is a periodic traveling wave of
the equation $u_t=u_{xx}-\beta(t)u_x+f_M(t,u)$. By the comparison principle we see that the solution $u$ of \eqref{p} satisfies
\begin{equation}\label{qqul}
u(t,x)\leqslant  Q_M \Big(t, x +\bar{c} t - \int_0^t \beta(s) ds + x_0\Big) \ \ \mbox{ for } t>0,\ x\in[g(t),h(t)],
\end{equation}
where $x_0>0$ is a large real number depending on $u_0$.

Now we give another estimate for $u$. Set $z:=\int_0^t \beta(s)ds-x$ and let $v$ be the solution of the Cauchy problem
\[
\left\{
\begin{array}{l}
v_t=v_{zz}+\frac{1}{M}f_M(t,Mv),\quad  t>0, \ z\in \R,\\
v(0,z) =\left\{\begin{array}{ll}
\frac{1}{M} u_0(-z), & |z|\leqslant h_0,\\
0, & |z|>h_0.
\end{array}
\right.
\end{array}
\right.
\]
It then follows from the comparison principle that
\begin{equation}\label{usma11}
u(t,x)\leqslant Mv\Big(t, \int_0^t \beta(s)ds-x\Big)\ \ \mbox{for }
(t,x)\in(0,\infty)\times[g(t),h(t)].
\end{equation}
On the other hand, it follows from \cite[Proposition 2.3]{HNRR}\footnote{Note that \cite[Proposition 2.3]{HNRR} holds
for the equation $u_t =u_{xx} + au$ with $a>0$ being a constant. If $a=a(t)$ is a periodic function with
$\bar{a} >0$, then the function $w:= ue^{-\int_0^t [a(s) -\bar{a}] ds}$ satisfies $w_t =w_{xx} + \bar{a}w$, and so the conclusions in \cite[Proposition 2.3]{HNRR} hold for $w$.} and its proof that there are positive constants $C_0$, $C_1$, $C_2$,
and $t_0$ depending on $u_0$ such that (see also \cite{GLZ})
\begin{equation}\label{usma112}
v\left(t, \bar{c}t-\frac{3}{\bar{c}}\ln \Big( 1+\frac{t}{t_0} \Big)+z\right)
\leqslant C_0 Z(t,z), \ \ t>0,\ z\geqslant h_0,
\end{equation}
where
\[
Z(t,z):=\frac{C_1}{\sqrt{t_0}}z e^{ -\frac{\bar{c}}{2} z}
\Big[C_2e^{\frac{-z^2} {4(t+t_0)}}+\xi(t,z)\Big],\quad t>0,\ z\in\R,
\]
with $\xi(t,z)$ satisfying
\[
\limsup_{t\to\infty} \sup_{0\leqslant z\leqslant \sqrt{t+1}}|\xi(t,z)|\leqslant \frac{C_2}{2}.
\]
Thanks to \eqref{usma11} and \eqref{usma112}, it is easy to check that
\begin{equation}\label{uss11}
u(t, x)\leqslant C_3 e^{-\frac{5\bar{c}}{12}(Y(t)-x)} \ \mbox{ for }
\max\{Y(t)-\sqrt{t}, g(t)\}\leqslant x\leqslant \min\{Y(t), h(t)\},\ t\gg 1,
\end{equation}
where $C_3$ is a positive constant and
\begin{equation}\label{Yt}
Y(t):= \bar{\beta} t - \bar{c} t + \frac{3}{\bar{c}} \ln t \quad \mbox{for } t>0.
\end{equation}

\medskip
Now we consider the case $\bar{\beta} \geqslant \bar{c}$ and prove the boundedness of $g_\infty$
and the locally uniform convergence  $u\to 0$.

\begin{lem}\label{lcr}
Assume $\bar{\beta}\geqslant\bar{c}$ and $(u,g,h)$ is a solution of
\eqref{p}. Then
\begin{itemize}
\item[(i)] for any $K \in I_\infty$, $\|u(t,\cdot)\|_{L^\infty ([g(t), K])} \to 0$ as $t\to\infty$;
\item[(ii)]   $g_\infty > -\infty$.
\end{itemize}
\end{lem}

\begin{proof}
(i). We prove the conclusion by showing that
\begin{equation}\label{utt0}
u(t,x)\leqslant C' t^{-\frac{5}{4}},\quad x\in[g(t), K],\ t\gg 1,
\end{equation}
for some $C'>0$.

{\it Case 1: $\bar{\beta} > \bar{c}$}. In this case, for $x\in [g(t), K]$ and $t\gg 1$, by
\eqref{qqul} and \eqref{qq1} we have
\begin{eqnarray*}
u(t,x) & \leqslant &  Q_M ( t, K +\bar{c}t - \bar{\beta} t +O(1) ) \leqslant C_1 (\bar{\beta}-\bar{c})t
e^{\frac{\bar{c}}{2} (\bar{c} -\bar{\beta})t}  \leqslant  C' t^{-\frac54},
\end{eqnarray*}
provided $C' >0$ is large enough, where $C_1 >0$ is a constant independent of $t$.

{\it Case 2: $\bar{\beta} =\bar{c}$}. In this case we have $Y(t) = \frac{3}{\bar{c}} \ln t$.

{\it Subcase 2.1:  $g(t) < Y(t)-\sqrt{t}$ for some $t \gg 1$}. In this case,
for $x\in[g(t),Y(t)-\sqrt{t}]\not= \emptyset$ and $t\gg 1$, by \eqref{qqul} and \eqref{qq1} we have
\begin{eqnarray*}
u(t,x) & \leqslant & Q_M (t, Y(t)-\sqrt{t} + O(1) )  =  Q_M \Big(  t, \frac{3}{\bar{c}}  \ln t -\sqrt{t} +O(1)\Big) \\
& \leqslant &  Q_M \Big( t, -\frac{3}{\bar{c}}  \ln t \Big)  \leqslant C_2  \ln t \cdot t^{-\frac{3}{2}}
\leqslant C' t^{-\frac{5}{4}},
\end{eqnarray*}
provided $C' >0$ is large, where $C_2$ is a constant independent of $t$.
For $x \in [Y(t)-\sqrt{t}, K]$ and $t\gg 1$, by \eqref{uss11} we have
\begin{equation}\label{g(t)-K}
u(t,x)\leqslant  C_3 e^{-\frac{5\bar{c}}{12}(Y(t)-K)}= C_3 e^{\frac{5\bar{c} K}{12}} \cdot t^{-\frac{5}{4}}.
\end{equation}

{\it Subcase 2.2: $Y(t) -\sqrt{t} <g(t)$ for some $t\gg 1$}.
In this case, for $x\in [g(t), K]$ and $t\gg 1$, the inequalities in \eqref{g(t)-K} remain valid, and so
\eqref{utt0} holds.

\smallskip

(ii). We now use the principal eigenfunction and the estimate \eqref{utt0}
to construct a suitable upper solution to prove $g_\infty>-\infty$.
Since $\bar{\beta}\geqslant\bar{c}$, for any given $\ell>0$, it is well known that
the problem \eqref{eigen-p} with $k(t)=-\beta(t)$, admits a unique positive principal eigenvalue $\lambda_1$,
whose corresponding positive eigenfunction is denoted by $\varphi_1(t,x)$.
Let $\zeta_l(t)$ be the leftmost local maximum point of $\varphi_1$, and set
\[
\zeta:=\min_{t\in [0,T]}\zeta_l(t),\  \epsilon:=\min_{t\in [0,T]}\varphi_1(t,\zeta),\ \mu^0:=\max_{t\in [0,T]}\mu(t),
\ D:=\max_{t\in [0,T]}(\varphi_1)_x(t,0),
\]
then
\[
\zeta>0,\ \ D>0 \ \mbox{ and }\ \  (\varphi_1)_x(t,x)> 0\ \mbox{ for } (t,x)\in[0,\infty)\times[0, \zeta ).
\]
Choose an integer $m_1$ such that $\tau_1 := m_1 T \geqslant 1+ \frac{5}{4\lambda_1}$ and that \eqref{utt0} holds
for $t>\tau_1$. Define
\[
\rho(t):=-g(\tau_1)+\zeta+4C'\mu^0D\epsilon^{-1} \big[ \tau_1^{-\frac{1}{4}}
-(t+ \tau_1)^{-\frac{1}{4}} \big]\ \mbox{ for } t\geqslant 0,
\]
\[
w(t,x):=C'\epsilon^{-1}(t+\tau_1)^{-\frac{5}{4}}\varphi_1
(t,x+\rho(t)) \ \mbox{ for } t\geqslant 0,\
x\in [-\rho(t), -\rho(t)+\zeta].
\]
A direct calculation yields that
$$
w_t-w_{xx}+\beta(t)w_x-f(t,w) \geqslant w\Big(\lambda_1-\frac{5}{4(t+\tau_1)}\Big)\geqslant 0, \quad t>0,\ x\in[-\rho(t), -\rho(t)+\zeta],
$$
\[
-\rho'(t)=-C'\mu^0D\epsilon^{-1}(t+\tau_1)^{-\frac{5}{4}}
\leqslant -\mu(t)w_x(t,-\rho(t)), \quad t\geqslant 0,
\]
and when $-\rho(t)+\zeta \in [g(t+\tau_1 ), h(t +\tau_1 )]$ for some $t>0$ we have
$$
w(t,-\rho(t)+\zeta)  =  C'\epsilon^{-1}(t+\tau_1)^{-\frac{5}{4}}\varphi_1 (t,\zeta)  \geqslant  C'(t+\tau_1)^{-\frac{5}{4}}
 \geqslant  u(t+\tau_1, -\rho(t)+\zeta ).
$$
Here we used the estimate \eqref{utt0} in the last inequality. Hence,
$(w, -\rho, -\rho +\zeta)$ is an upper solution of \eqref{p}, and so
by the comparison principle (cf. \cite[Lemma 2.2]{DuLou}) we have
\[
g(t+\tau_1) \geqslant  -\rho(t) >  g(\tau_1)-\zeta-4C'\mu^0D
\epsilon^{-1}\tau_1^{-\frac{1}{4}}>-\infty,\quad t \geqslant 0.
\]
This completes the proof of the lemma.
\end{proof}

\begin{remark}\rm
We remark that this lemma is an analogue of Lemma 4.1 and Proposition 4.6 in \cite{GLZ}.
But our construction for upper solutions is more complicated. The difficulty is that in the present
case we have no monotonicity for $u(t,\cdot)$ in the interval $[g(t), -h_0]$ (since
$\beta(t)$ may change sign).
The boundedness of $g(t)$ indicates that when the advection intensity is large, namely, when
$\bar{\beta}\geqslant\bar{c}$, spreading does not happen.
\end{remark}

\begin{remark}\label{rem:h finite u to 0}\rm
The conclusion (i) in the previous lemma also  indicates that if $h_\infty <\infty$, then
$\|u(t,\cdot)\|_{L^\infty ([g(t),h(t)])}$ $\to 0$ as $t\to \infty$.
Hence $h_\infty <\infty$ is a necessary and sufficient condition for vanishing in case $\bar{\beta} \geqslant \bar{c}$.
We will give other sufficient conditions for vanishing in the next two lemmas.
\end{remark}

Next let us consider the case $\bar{\beta} \geqslant B(\tilde{\beta}) $ and prove the boundedness of $h_\infty$ and the uniform convergence $u\to 0$.

\begin{lem}\label{boundh}
Assume $\HH$, $\HHH$, $\bar{\beta} \geqslant B(\tilde{\beta})$ and that $(u,g,h)$ is a
solution of \eqref{p}. Then
\begin{itemize}
\item[(i)]  $\|u(t,\cdot)\|_{L^\infty ([g(t), h(t)])} \to 0$ as $t\to\infty$;
\item[(ii)]   $h_\infty <\infty$.
\end{itemize}
\end{lem}

\begin{proof}
(i). We prove the conclusion by showing that
\begin{equation}\label{est u large beta}
u(t,x)\leqslant C'' t^{-\frac{5}{4}},\quad x\in[g(t), h(t)],\ t\gg 1,
\end{equation}
for some $C''>0$. We will use a conclusion in Remark \ref{rem:est h all beta} in Section 5
(whose proof is independent of the current conclusions), which says that
\begin{equation}\label{hbound00}
h(t)\leqslant \bar{r}t+H_r\ \mbox{ for all }t\geqslant0 \mbox{ and some } H_r >0.
\end{equation}

{\it Case 1: $\bar{\beta} > B(\tilde{\beta})$}. In this case we have $\kappa:=\bar{\beta}-\bar{c}-\bar{r}>0$
by the definition of $B(\theta)$. Thanks to \eqref{qqul}, \eqref{hbound00} and \eqref{qq1}, we have
$$
u(t,x)  \leqslant   Q_M ( t, h(t) +\bar{c}t - \bar{\beta} t +O(1) ) \leqslant Q_M (t, -\kappa t +O(1))
 \leqslant  C_4 t e^{- \frac{\bar{c} \kappa }{2} t}  \leqslant  C'' t^{-\frac54},
$$
provided $C'' >0$ is large enough, where $C_4 >0$ is a constant independent of $t$.

{\it Case 2: $\bar{\beta} = B(\tilde{\beta} )$}. In this case we have
$$
Y(t) -\sqrt{t} = (\bar{\beta} -\bar{c})t + \frac{3}{\bar{c}}\ln t -\sqrt{t} +O(1) >g(t),\quad t\gg 1.
$$
For $x\in[g(t),Y(t)-\sqrt{t}] $ and $t\gg 1$, by \eqref{qqul} and \eqref{qq1} we have
\begin{eqnarray*}
u(t,x) & \leqslant & Q_M (t, Y(t)-\sqrt{t} +\bar{c} t - \bar{\beta} t + O(1) )  =  Q_M \Big(  t, \frac{3}{\bar{c}}  \ln t -\sqrt{t} +O(1)\Big) \\
& \leqslant &  Q_M \Big( t, -\frac{3}{\bar{c}}  \ln t \Big)  \leqslant C_5  \ln t \cdot t^{-\frac{3}{2}}
\leqslant C'' t^{-\frac{5}{4}},
\end{eqnarray*}
provided $C'' >0$ is large, where $C_5$ is a constant independent of $t$.
For $x \in [Y(t)-\sqrt{t}, h(t)]$ and $t\gg 1$, by \eqref{uss11} and \eqref{hbound00} we have
\begin{equation}\label{g(t)-h(t)}
u(t,x)\leqslant  C_3 e^{-\frac{5\bar{c}}{12}(Y(t)- h(t))} \leqslant C'' t^{-\frac{5}{4}},\quad \mbox{provided } C'' >0 \mbox{ is large enough}.
\end{equation}

\medskip

(ii). Based on the estimate \eqref{est u large beta} for $u$, we now construct
an upper solution to prove $h_\infty <\infty$. Since $\bar{\beta}\geqslant B(\tilde{\beta})$, for any given $\ell>0$, the problem \eqref{eigen-p} with $k(t)=-\beta(t)$ admits a unique positive principal eigenvalue $\lambda_1$, whose corresponding positive eigenfunction is denoted by
$\varphi_1(t,x)$. Let $\zeta_r (t)$ be the rightmost local maximum point
of  $\varphi_1(t,\cdot)$, and set $\zeta_*:=\max_{t\in [0,T]}\zeta_r (t),\ \epsilon_*:=\min_{t\in [0,T]}\varphi_1(t,\zeta_*)$
and $D_*:=\max_{t\in [0,T]} |\varphi_{1x} (t,\ell)|$, then
\[
D_* >0,\ \ \ (\varphi_1)_x(t,x)< 0\ \mbox{ for } (t,x)\in[0,\infty)\times (\zeta_*,\ell].
\]
Choose an integer $m_2$ such that $\tau_2  := m_2 T > 1 + \frac{5}{4\lambda_1}$ and that \eqref{est u large beta}
holds for $t>\tau_2$. Denote $\mu^0:=\max_{t\in [0,T]}\mu(t)$,
\[
\rho_*(t):=h(\tau_2)+\zeta_* + 4C'' \mu^0 D_*\epsilon^{-1}_*
[\tau_2^{-\frac{1}{4}}-(t+ \tau_2)^{-\frac{1}{4}}]\ \mbox{ for } t\geqslant 0,
\]
and
\[
w_*(t,x):=C'' \epsilon_*^{-1}(t+ \tau_2)^{-\frac{5}{4}}\varphi_1 (t,x-\rho_*(t)) \ \mbox{ for } t\geqslant 0,\
x\in [\rho_*(t)+\zeta_*, \rho_*(t)+\ell].
\]
A direct calculation as in the proof of the previous lemma shows
that $(w_*, \rho_*+\zeta_*, \rho_*+\ell)$ is an upper solution to
\eqref{p}, and so by comparison we have
\[
h(t+\tau_2) < \rho_* (t) +\ell < \ell +  h(\tau_2)+ \zeta_* +  4C'' \mu^0 D_*\epsilon^{-1}_*\tau_2^{-\frac{1}{4}}<\infty,
\quad t\geqslant 0.
\]
This proves the lemma.
\end{proof}

\smallskip

\noindent
{\it Proof of Theorem \ref{thm:large beta}.}
The conclusions in Theorem \ref{thm:large beta} follow from Lemma \ref{boundh} immediately.
\hfill
$\square$

\subsection{Problem with medium-sized advection: $\bar{c}\leqslant \bar{\beta} <B(\tilde{\beta})$}\label{bbbmd}
In this subsection we consider the case with medium-sized advection. New phenomena like virtual spreading
and transition happen for some solutions.
In the first part we give some conditions for vanishing and for
virtual spreading, in the second part we prove Theorem \ref{thm:middle beta}.

\subsubsection{Vanishing and virtual spreading phenomena}\label{vvsp}
When $\bar{\beta}\geqslant\bar{c}$, it follows from Lemma \ref{lcr} that $g_\infty>-\infty$ and
$u\to 0$ in $[g(t), K]$ for any $K \in I_\infty$. We now present a sufficient condition
for vanishing.

\begin{lem}\label{mavan}
Assume that $\bar{c}\leqslant\bar{\beta}<B(\tilde{\beta})$, and $(u,g,h)$ is
the solution of \eqref{p}. Then vanishing happens when $\|u_0\|_{L^\infty([-h_0,h_0])}$
is sufficiently small.
\end{lem}

\begin{proof}
For any fixed $h_1 >h_0$, we use $\lambda_1 $ to denote the principal eigenvalue
of the problem \eqref{eigen-p} with $\ell=2 h_1 $ and $k(t)=-\beta(t)$.
Since $\bar{\beta}\geqslant \bar{c}$, we have $\lambda_1 >0$ by Lemma \ref{lem:1eigenvalue}.
The rest of the proof is exactly the same as that for Lemma \ref{vfsma}. We omit the detail.
\end{proof}

Next we give a sufficient condition for virtual spreading, which means that spreading
should be considered in a moving frame.

\begin{lem}\label{mvirtsp}
Assume that $\bar{c}\leqslant\bar{\beta}<B(\tilde{\beta})$ and  $(u,g,h)$
is a solution of \eqref{p}. Then virtual spreading happens if and only if,
there exist $x_i,\ \epsilon_i, \ \ell_i$ ($i=1,2$) with $0<\epsilon_1 <\epsilon_2\ll 1$ and an integer
$m \geqslant 0$ such that
\begin{equation}\label{ssuso1}
u(mT ,x)\geqslant W^{\epsilon_i} (0, x-x_i) \ \ \mbox{ for }\ x\in [x_i -\ell_i ,\, x_i],\quad i=1,2,
\end{equation}
where, for each $i$, $W^{\epsilon_i}$ is the compactly supported traveling wave given in
Proposition \ref{prop:tw compact large}(ii).
\end{lem}

\begin{proof}
Clearly the inequality \eqref{ssuso1} is a consequence of virtual spreading (see the definition in Section 2). We only
need to show that \eqref{ssuso1} is a sufficient condition for virtual spreading.

For each $i=1,2$, from Proposition \ref{prop:tw compact large}(ii) we know that $R^{\epsilon_i}(t):=
\int_0^t r^{\epsilon_i}(s) ds$ (with $r^{\epsilon_i}(t):= r(t;\beta)-\epsilon_i$) satisfies
$(R^{\epsilon_i})' (t) <  -\mu(t) W^{\epsilon_i}_x (t, R^{\epsilon_i}(t))$. Hence $W^{\epsilon_i}(t,x)$ is
a lower solution of \eqref{p}, and by \eqref{ssuso1} we have
\begin{equation}
u(t+mT ,x)> W^{\epsilon_i} (t, x-x_i) \ \ \mbox{ for }\ x\in [R^{\epsilon_i}(t)+ x_i -\ell_i ,\, R^{\epsilon_i}(t) +x_i],\ t>0,
\end{equation}
and $R^{\epsilon_i}(t) + x_i < h(t +mT)$ for $t> 0$. Define
$$
w(t,x):= u(t+mT, x +R^{\epsilon_2}(t) +x_2)\mbox{ for } G(t)\leqslant x\leqslant H(t),\ t\geqslant 0,
$$
with
$$
H(t):= h(t+mT) - R^{\epsilon_2}(t) -x_2 \mbox{ and } G(t):= g(t+mT) - R^{\epsilon_2}(t) -x_2 \mbox{ for } t\geqslant 0.
$$
Then as $t\to\infty$, $G(t)\to-\infty$,
$$
H(t) = h(t+mT) - R^{\epsilon_2}(t) -x_2 > R^{\epsilon_1}(t) +x_1 - R^{\epsilon_2}(t) -x_2 \to \infty,
$$
and $w$ satisfies
$$
w_t=w_{xx}-[\beta(t)-r^{\epsilon_2}(t)]w_x+ f(t,w), \quad t>0,\ G(t)<x<H(t).
$$
Since $\bar{\beta}<B(\tilde{\beta})$, by the definition of $B(\theta)$ and Lemmas \ref{lem:B(theta)} and
\ref{lem:property of beta*} we have $\overline{r(t;\beta)}>\bar{\beta} -\bar{c}$, and so $\overline{r^{\epsilon_2}} =
\ov{r(t;\beta)}  -\epsilon_2 > \bar{\beta} - \bar{c}$ provided $ \epsilon_2 >0$ is small. On the other hand,
we have $\overline{r^{\epsilon_2}} < \bar{\beta} +\bar{c}$ by Proposition \ref{prop:tsw right}. Therefore,
$| \bar{\beta} - \overline{r^{\epsilon_2}}| <\bar{c}$ and it follows from Remark \ref{rem12} that $w(t,x)-P(t)\to 0$ as $t\to \infty$ locally uniformly
in $x\in \R$. Namely, virtual spreading happens for $u$.
\end{proof}

\subsubsection{Proof of Theorem \ref{thm:middle beta}}\label{mddsa1}
For any given $h_0>0$ and $\phi\in \mathscr {X}(h_0)$, we write the solution $(u,g,h)$ as $(u(t,x;\sigma\phi),
g(t;\sigma\phi), h(t;\sigma\phi))$ to emphasize the dependence on the initial data $u_0=\sigma\phi$. Define
$$
\Sigma_0:=\{\sigma\geqslant 0 : \mbox{vanishing happens for } u(\cdot, \cdot; \sigma \phi)\},\  \ \sigma_*:=\sup \Sigma_0.
$$
It follows from Lemma \ref{mavan} that $\sigma\in \Sigma_0$ for all small $\sigma>0$, thus $\Sigma_0$
is nonempty. By comparison $[0,\sigma_*)\subset \Sigma_0$. If $\sigma_*=\infty$, then there
is nothing left to prove. In what follows we consider  the case $\sigma_* \in(0,\infty)$.
Using Lemma \ref{mavan} one can prove $\sigma_* \not\in \Sigma_0$ in a similar way as
proving Theorem 4.9 in \cite{GLZ}. Hence $\Sigma_0 = [0,\sigma_*)$.

On the other hand, define
$$
\Sigma_1:=\{\sigma:\ \mbox{virtual spreading happens for } u(\cdot, \cdot; \sigma\phi)\},\ \ \sigma^*:=\inf \Sigma_1.
$$
When $\Sigma_1 = \emptyset $, virtual spreading does not happen for any $\sigma\geqslant0$. Then each solution
$u(t,x;\sigma\phi)$ with $\sigma\in [\sigma_*, \infty)$  is a transition one. When $\sigma^* < \infty$, it
is easy to see from Lemma \ref{mvirtsp} and the continuous dependence of the solution on the initial data
that $\Sigma_1$ is an open set. So $\Sigma_1 = (\sigma^*,\infty)$. Each solution $u(t,x;\sigma\phi)$
with $\sigma\in [\sigma_*, \sigma^*]$  is a transition one, for which neither
virtual spreading nor vanishing happens. Moreover, for any transition solution, it follows from Lemma \ref{lcr} and Remark
\ref{rem:h finite u to 0} that $g_\infty >-\infty$, $h_\infty =\infty$ and $u(t,\cdot)\to 0$ as $t\to \infty$ locally
uniformly in $(g_\infty, \infty)$.

This completes the proof of Theorem \ref{thm:middle beta}. \hfill \qed

\begin{remark}\rm
For the homogeneous problem (that is, $\beta, f$ and $\mu$ are independent of $t$),
in the medium-sized advection case: $\beta\in (\bar{c}, B(0))$, it was shown
in \cite[Theorem 2.2]{GLZ} that transition happens for exactly one initial data, that is,
$\sigma_* =\sigma^*$, and any transition solution $u$ converges to a tadpole-like traveling
semi-wave. It has a big head on the right side and a long tail on the left side,
and moves rightward with speed $\beta -\bar{c}$. We guess that similar results
should be true for our time periodic problem \eqref{p}. We will study this problem in a forthcoming paper.
The approach should be much more complicated than the homogeneous case since we can not construct
a time periodic tadpole-like semi-wave beforehand. In homogeneous case, it was constructed
easily by using a phase plane analysis, and then it played a key role in the later approach
for proving the convergence of the transition solution to this semi-wave.
\end{remark}


\section{Asymptotic profiles of (virtual) spreading solutions}
In this section we study the asymptotic profiles for spreading or virtual spreading solutions.
As before, we write $U(t,R(t)-x;\beta -r)$ as the rightward periodic traveling semi-wave with speed $r(t;\beta)$
and $R(t):= \int_0^t r(s;\beta) ds$, write $U(t,x+L(t);-\beta-l)$ as the leftward periodic traveling
semi-wave with speed $l(t;\beta)$ and $L(t):= \int_0^t l(s;\beta) ds$.
In the first subsection we show that $|h(t)-R(t)|$ is bounded
when $0\leqslant\bar{\beta}<B(\tilde{\beta})$, and $|g(t)+L(t)|$
is bounded when $0\leqslant\bar{\beta} <\bar{c}$. Based on these results, we prove Theorem
\ref{thm:profile of spreading sol} in the second subsection.

\subsection{Boundedness for $|h(t)-R(t)|$ and $|g(t)+L(t)|$.}\label{sub51}
For convenience, we normalize the problem \eqref{p} by setting
\begin{equation}\label{bihu}
v(t,x):=\frac{u(t,x)}{P(t)},\ \ \gamma(t):=\mu(t)P(t).
\end{equation}
Then the problem \eqref{p} is converted into
\begin{equation}\label{fuzhu1}
\left\{
\begin{array}{ll}
v_t=v_{xx}-\beta(t)v_x+ F(t,v), & t >0,\ g(t)<x<h(t),\\
v(t,x)= 0,\ \
g'(t)=-\gamma(t) v_x(t,x),&  t >0,\ x= g(t),\\
v(t,x)= 0,\ \
h'(t)=- \gamma (t)v_x(t,x),&  t >0,\ x= h (t),\\
v(0,x)=u_0(x) /P(0), &  -h_0\leqslant x\leqslant h_0,
\end{array}
\right.
\end{equation}
where the new nonlinearity $F(t,v):=\frac{1}{P(t)}[f(t, P(t)v)-f(t, P(t))v]$ satisfies
$$
\left\{
 \begin{array}{l}
 F(t,v)\in C^{\nu/2, 1+\nu/2} ([0,T]\times \R) \mbox{ for some } \nu\in (0,1),\ T\mbox{-periodic in } t,\
 F(t,0)\equiv F(t,1)\equiv 0, \\
  \mbox{for any } t\in [0,T],\ F(t,v)>0 \mbox{ for } 0<v<1,\ \
  F(t,v)<0 \mbox{ for } v>1,  \\
  F(t,v)/v \mbox{ is decreasing in } v >0,\\
  a_1(t) := F_v(t,0) =a(t) - f(t,P(t))/P(t),\ \ \alpha_1(t):= F_v(t,1)= \alpha(t)/P(t).
  \end{array}
 \right. \ \hskip 10mm \hfill
$$
Clearly, $\ov{a_1} = \bar{a}$ and, by the condition $\HHH$, $\alpha_1(t) <-2 \delta$ for some $\delta>0$.
The latter inequality implies that, for some small $\varepsilon >0$, there holds
\begin{equation}\label{fuzhu3}
F_v(t,v) \leqslant -\delta \    \mbox{ for }\ t\in[0,T],\ v\in[1-\varepsilon ,1+\varepsilon ].
\end{equation}

We first give a rough estimate for $g$ and $h$, and show that $u(t,\cdot) \to P(t)$
in the interior of the ``main" habitat of $u$.

\begin{lem}\label{lem:rough gh}
\begin{itemize}
\item[(i)]
Assume that $0\leqslant \bar{\beta} <\bar{c}$ and that spreading happens for $(u,g,h)$. Then for any
constants $c_1, c_2$ satisfying $-\bar{l} < -c_1 < 0 < c_2 <\bar{r}$, there exists a constant $K_1>0$ such that
\begin{equation}\label{u to P(t)}
g(t) < -c_1 t,\quad c_2 t <h(t),\quad \|u(t,\cdot)-P(t)\|_{L^\infty ([-c_1 t, c_2 t])} \leqslant K_1 e^{-\delta t} \mbox{ for } t\gg 1.
\end{equation}

\item[(ii)] Assume that $\bar{c} \leqslant \bar{\beta} <B(\tilde{\beta}) $ and that virtual spreading happens for $(u,g,h)$.
Then for any constants $c_3, c_4$ satisfying $\bar{\beta} -\bar{c} < c_3 < c_4 <\bar{r}$, there exists a constant $K_2>0$ such that
\begin{equation}\label{u to P(t) virtual}
c_4 t <h(t),\quad \|u(t,\cdot)-P(t)\|_{L^\infty ([c_3 t, c_4 t])} \leqslant K_2 e^{-\delta t} \mbox{ for } t\gg 1.
\end{equation}
\end{itemize}
\end{lem}

\begin{proof}
(i) For any $\epsilon>0$ satisfying $c_1 <\bar{l} -\epsilon$ and $c_2<\bar{r}-\epsilon$, when
it is sufficiently small, by Proposition \ref{prop:tw compact large} (ii) and (iii), the function
$W^{\epsilon} (t,x)$ is a compactly supported traveling wave which travels rightward with
speed $r^{\epsilon}(t):= r(t;\beta) -\epsilon$, and the function $W^{\epsilon}_l(t,x)$
is a compactly supported traveling wave which travels leftward with
speed $l^{\epsilon}(t):= l(t;\beta) -\epsilon$.
Since  spreading happens for the solution $u$, there is a large integer $m$ such that
both $W^{\epsilon}(0,x)$ and $W^{\epsilon}_l (0,x)$ lie below $u(mT, x)$. Hence $W^{\epsilon}(t,x)$ and $W^{\epsilon}_l (t,x)$
lie below $u(mT+t, x)$ for all $t>0$ since they are lower solutions of \eqref{p}. So
$$
g(t+mT) < - \int_0^t l^\epsilon (s) ds +O(1) ,\quad   \int_0^t r^\epsilon (s) ds +O(1) < h(t +mT),\quad t\geqslant 0.
$$
The first two inequalities in \eqref{u to P(t)} then follows.

By \eqref{fuzhu3}, one can use the lower solutions $W^{\epsilon}(t,x)$ and $W^{\epsilon}_l (t,x)$ and
use the same argument for the normalized function $v$ as in the proof of \cite[Lemma 6.5]{DuLou} to show that
$$
|v(t,x) - 1| \leqslant k_1 e^{-\delta t} \mbox{ for } x\in [-c_1 t, c_2 t],\ t\gg 1,
$$
where $k_1>0$ is a constant. This reduces to the third inequality in \eqref{u to P(t)}.

(ii) Choose $\epsilon_1 \in (0,\bar{c})$ small such that $c_5:= \bar{\beta} -\bar{c} + \epsilon_1 <
c_3 < c_4 < \bar{r} -\epsilon_1 $. Then $\bar{\beta} -c_5 = \bar{c} -\epsilon_1$ and
$\bar{\beta} -\bar{r} +\epsilon_1  <\bar{c} -\epsilon_1$. On the other hand, by Proposition \ref{prop:tsw right}
we have $0<\bar{r}<\bar{\beta} +\bar{c}$, and so $\bar{\beta} -\bar{r} +\epsilon_1  > -\bar{c} +\epsilon_1$.
Thus both $U_0(t,z; \beta -c_5,\ell)$ and $U_0(t, z; \beta - r(t;\beta) +\epsilon_1, \ell)$ exist
when $\ell$ is large. Define
$$
W_1 (t,x):= U_0(t,c_5 t -x; \beta -c_5,\ell) \mbox{ and } W_2 (t,x):= U_0(t, R(t) -
\epsilon_1 t -x; \beta - r(t;\beta) +\epsilon_1, \ell).
$$
Then they are compactly supported traveling waves of \eqref{p}$_1$. $W_2$ moves rightward with average
speed $\bar{r}-\epsilon_1$ and $W_1$ moves rightward with constant speed $c_5$.
Since virtual spreading happens for $u$, both $W_1(0,x)$ and $W_2(0,x)$ can lie below
$u(m'T, x+x'')$ for some large integer $m'$ and some suitable shift $x''$. By comparison
they stay below $u$ since they are lower solutions of the problem \eqref{p}.
In fact, at the right end
point $x= R(t) -\epsilon_1 t$ of $W_2$ we have $ (R(t)-\epsilon_1 t)' = r(t;\beta)-\epsilon_1 <
-\mu(t) W_{2x} (t,R(t) - \epsilon_1 t) = A_0[\beta - r+\epsilon_1, \ell]$ by Proposition \ref{prop:tw compact large}(ii).
So $W_2$ is a lower solution and $R(t) -\epsilon_1 t < h(t+m' T)$ for $t>0$.
Hence $c_4 t<h(t)$ for $t\gg 1$. For the function $W_1$, though it may not satisfy the free boundary
condition on its right end point $x= c_5 t$, it is still a lower solution since it does not touch $u$ at this
end point at all (it moves rightward behind $W_2$).
Now we can use these lower solutions to support the spreading of $u$ and using the same argument
for the normalized function $v$ as in the proof of \cite[Lemma 6.5]{DuLou} to show that
$$
|v(t,x) - 1| \leqslant k_2 e^{-\delta t} \mbox{ for } x\in [c_3 t, c_4 t],\ t\gg 1,
$$
where $k_2>0$ is a constant. This reduces to the second inequality in \eqref{u to P(t) virtual}.
\end{proof}

Next we prove the boundedness of $h(t)-R(t)$ and show that $u(t,\cdot) \approx P(t)$ in the domain
$[c_l t, h(t)-X]$, where $X>0$ is a large number and, for any given small $\epsilon_0 >0$, $c_l$ denotes a
speed defined by
\begin{equation}\label{def:cl}
c_l :=
\left\{
\begin{array}{ll}
  0, & \mbox{when } 0\leqslant \bar{\beta} <\bar{c} \mbox{ and spreading happens},\\
  \bar{\beta} -\bar{c}+\epsilon_0, & \mbox{when } \bar{c}\leqslant \bar{\beta} <B(\tilde{\beta}) \mbox{ and virtual spreading happens}.
\end{array}
\right.
\end{equation}

\begin{prop}\label{pro:sigma01}
Assume that $\HH$, $\HHH$ and $\beta\in\PP$ satisfies $0\leqslant\bar{\beta}<B(\tilde{\beta})$.
Assume further that spreading or virtual spreading happens for the solution $(u,g,h)$. Then
\begin{itemize}
\item[(i)] there exists $C>0$ such that
\begin{equation}\label{hh1}
|h(t)-R(t) |\leqslant C \ \ \mbox{for all } t\geqslant0 ;
\end{equation}

\item[(ii)] for any small $\epsilon, \epsilon_0>0$, let $c_l$ be the number defined by \eqref{def:cl}, then
there exists $X_\epsilon >0$ and $T_\epsilon >0$ such that
\begin{equation}\label{u to P near 0}
\|u(t,\cdot ) - P(t) \|_{L^\infty ([c_l t, h(t) -X_\epsilon])} \leqslant \epsilon \ \mbox{ when } t> T_\epsilon.
\end{equation}
\end{itemize}
\end{prop}

\begin{proof}
As above we normalize the periodic rightward traveling semi-wave $U(t,z; \beta-r)$ by
$$
V (t,z):=\frac{U(t,z; \beta -r)}{P(t)}, \ \
$$
Then the problem \eqref{p-sw-k} is converted into
\begin{equation}\label{fuzhu2}
\left\{
\begin{array}{ll}
V_t=V_{zz}+[\beta(t)-r(t;\beta)]V_z+ F(t,V), & t\in[0,T],\ z>0,\\
V(t,0)= 0,\ \ V(t,\infty)=1, &  t\in[0,T],\\
V(0,z)=V(T,z),\ \ V_z(t,z)>0 &  t\in [0,T], \ z>0,\\
r(t;\beta)=\gamma(t)V_z(t,0),  & t\in[0,T],\ z>0,\\
\end{array}
\right.
\end{equation}
where $\gamma$ and $F$ are the same as those in \eqref{bihu} and \eqref{fuzhu1}.
\smallskip

{\bf Step 1}. To give some upper bounds for $h(t)$ and $v(t,x)$.

First, we give a simple upper bound for $v(t,x)$. Let $\eta(t)$ be the solution of
$\eta_t=F(t,\eta)$ with initial value $\eta(0)=M/P(0) +1$, where $M$ is given by \eqref{def:M}.
Due to $F(t,v)<0$ for $v>1$, the function $\eta(t)$ decreases to $1$ as $t\to \infty$. Hence,
for $\varepsilon>0$ in \eqref{fuzhu3}, there exists a large integer $m $ such that,
$1<\eta(t)<1+\varepsilon $ when $t \geqslant m T$.
By \eqref{fuzhu3} we have $\eta_t = F(t,\eta) \leqslant \delta (1-\eta)$ for $t\geqslant m T$.
So, $\eta (t)\leqslant 1+ \varepsilon  e^{\delta m T }  e^{-\delta t}$ for $t \geqslant m T$.
Clearly $\eta(t)$ is an upper solution of \eqref{fuzhu1}, and so
\begin{equation}\label{Vtoto1}
v(t,x) \leqslant \eta (t)\leqslant 1+ \varepsilon  e^{\delta m T}  e^{-\delta t} \ \mbox{ for }
g(t)\leqslant x\leqslant h(t),\ t \geqslant m T.
\end{equation}
Take an integer $m' >m$ such that $  e^{\delta (m -m') T }< 1 /2$.
Since $V(t,\infty)= 1$, we can find $X >0$ such that, with $T'  := m' T$ and $M' := 2 \varepsilon  e^{\delta mT}$,
\begin{equation}\label{U1a1}
(1+M'e^{-\delta T'  })V(t,z )\geqslant 1+ \varepsilon  e^{\delta mT} e^{-\delta T' }
\ \mbox{ for all } t\in[0,T],\ z\geqslant X.
\end{equation}

Now we construct another finer upper solution $(v^+ ,g, h^+ )$ to \eqref{fuzhu1} as follows.
\[
h^+ (t): =\int_{T' }^t r(s;\beta)ds+ h (T' )+ K M'(e^{-\delta T' }-e^{-\delta t})+X \ \
\mbox{ for } t\geqslant T' ,
\]
\[
v^+ (t,x):=(1+M'e^{-\delta t})V(t,h^+ (t)-x)\  \ \mbox{ for } t\geqslant T' ,\ x\leqslant h^+ (t),
\]
where $K$ is a positive constant to be determined below.
Clearly, for all $t\geqslant T' $, $v^+ (t, g(t))>0= v(t, g(t))$,
$v^+ (t, h^+ (t))=0$, and
\begin{eqnarray*}
-\gamma(t)v^+ _x(t,h^+ (t))& = & \gamma (t)(1+M'e^{-\delta t})V_z(t,0)=(1+M'e^{-\delta t})r(t;\beta), \\
& < & r(t;\beta)+M' K \delta e^{-\delta t} = (h^+) '(t),
\end{eqnarray*}
if we choose $K$ with $K\delta > \max_{t\in [0,T]}r(t;\beta) $.  By the definition of $h^+ $ we
have $h (T' )<h^+ (T' )$. By \eqref{Vtoto1} and \eqref{U1a1} we have
\begin{eqnarray*}
v^+ (T' ,x) &  = & (1+M'e^{-\delta T' })V(T',  h (T' )+X -x) \geqslant
(1+M'e^{-\delta T' })V(T',X ) \\
& \geqslant & 1+ \varepsilon   e^{\delta mT} e^{ - \delta T' }\geqslant v(T' ,x),
\quad x\in[ g (T' ),h (T' )].
\end{eqnarray*}
We now show that
\begin{equation}\label{u+ upper}
 \mathcal{N} v^+ := v^+_t - v^+_{xx} + \beta(t)v^+_x -F(t,v^+) \geqslant 0,\quad x\in [g(t), h^+(t)],\ t>T' .
\end{equation}
In fact, by  a direct calculation we have
\begin{eqnarray*}
\mathcal{N} v^+ & = & -\delta M'e^{-\delta t}V+(1+M'e^{-\delta t}) \{KM'\delta e^{-\delta t}V_z+F(t,V)\} -F(t,(1+M'e^{-\delta t})V)\\
 & = & M'e^{-\delta t}\Big\{F(t,V)+K \delta (1+M'e^{-\delta t})V_z -\delta V\Big\} + F(t,V)- F(t,(1+M'e^{-\delta t})V)\\
 & = & \mathcal{F}:= M'e^{-\delta t}\Big\{F(t,V)+K \delta (1+M'e^{-\delta t})V_z-[F_v(t,(1+ \rho M'e^{-\delta t})V)+\delta]V\Big\}.
\end{eqnarray*}
for some $\rho \in (0,1)$.  Since $V(t,z)\to 1$ as $z\to \infty$ for all $t\in[0,T]$, there is $z_0>0$ such that
\[
V(t,z)\geqslant 1-\varepsilon \ \ \mbox{ for } (t,z)\in[0,\infty)\times(z_0,\infty).
\]
When $h^+ (t)-x>z_0$ and $t> T' $, we have $\mathcal{F}>0$ by \eqref{fuzhu3} and the fact that $M'e^{-\delta t}\leqslant\varepsilon $ for $t> T' $.
When  $0\leqslant h^+ (t)-x\leqslant z_0$ and $t> T' $, we have
$$
\mathcal{F}  \geqslant  M'e^{-\delta t}(K \delta D_1 - D_2 -\delta)\geqslant 0,\quad \mbox{ provided } K>0 \mbox{ is sufficiently large},
$$
where
$$
D_1:=\min_{(t,z)\in[0,T]\times [0,z_0]}V_z(t,z)>0 \mbox{ and } D_2:=\max_{(t,s)\in[0,T] \times[0,1+M']}F_v(t, s).
$$

Summarizing the above results we see that $(v^+, g, h^+)$ is an upper solution of \eqref{fuzhu1}.  By the comparison principle we have
$$
h(t) \leqslant h^+(t) \mbox{ for } t>T'\quad \mbox{and} \quad  v(t,x)\leqslant v^+(t,x) \leqslant 1+M' e^{-\delta t},\quad x\in [g(t), h(t)],\ t>T'.
$$
By the definition of $h^+$ we see that, for $H_r := h(T' )+X +KM'$, we have
\begin{equation}\label{hbd}
h(t)< R(t) +H_r \ \mbox{ for all } t\geqslant 0.
\end{equation}
For any $\epsilon>0$ and for $P^0 := \max_{t\in [0,T]} P(t)$, if we choose $T_1(\epsilon) >T'$ large such that
$P^0 M' e^{-\delta T_1(\epsilon)} < \epsilon$, then by the definition of $v^+$ we have
\begin{equation}\label{v<1+epsilon/P}
v(t,x)\leqslant v^+(t,x) \leqslant 1 +  \epsilon/P^0 ,\quad x\in [g(t), h(t)],\ t> T_1(\epsilon),
\end{equation}

\smallskip

{\bf Step 2}. To give some lower bounds for $h$ and $v(t,x)$.

For the number $\epsilon_0>0$ given in the assumption, define $c_l$ by \eqref{def:cl} and define $c_r:= \bar{r} -\epsilon_0$.
For the constant $\delta$ given in \eqref{fuzhu3}, it follows from Lemma \ref{lem:rough gh} that there
exist $K_1 >0$ and an integer $m'' >0$ such that
$$
c_rt\leqslant h(t),\ \quad \|v(t,\cdot)-1\|_{L^\infty ([c_l t, c_r t])} \leqslant K_1 e^{-\delta t},\quad t\geqslant T'' := m'' T.
$$
Define
\[
g^- (t) := c_l t ,\quad  h^- (t) :=\int_{T''}^t r(s;\beta)ds - K_2 K_1 (e^{-\delta  T''}-e^{-\delta  t}) + c_r T''\ \mbox{ for } t\geqslant T'',
\]
\[
v^- (t,x) :=(1- K_1  e^{-\delta  t})V(t, h^- (t)-x)\ \mbox{ for } x\in[g^- (t),h^- (t)],\ t\geqslant T''.
\]
Then for a suitable constant $K_2>0$, a similar argument as in Step 1 shows that $(v^- ,g^- , h^- )$ is a lower solution.
By the comparison principle we have
\begin{equation}\label{h>h->R}
h(t) \geqslant h^-(t)  \mbox{ for } t >T'',
\quad v(t,x)\geqslant v^- (t,x) \mbox{ for } x\in [g^- (t), h^-(t)],\ t>T''.
\end{equation}
Hence
\begin{equation}\label{hbd2}
h(t)\geqslant h^- (t) - \max_{t\in[0,T'']}|h(t)-h^- (t)| \geqslant R(t) -H_l \ \ \mbox{ for all } t\geqslant 0,
\end{equation}
where $H_l = \max_{t\in[0,T'']}|h(t)-h^- (t)|+\bar{r}T'' +K_2 K_1$.
Combining with \eqref{hbd} we obtain \eqref{hh1}.

On the other hand, for any $\epsilon>0$, since $V(t,\infty) =1$, there exists $X_1(\epsilon)>0$ such that
$$
V(t,z)> 1- \epsilon/(2P^0) \mbox{  for } t\in [0,T],\ z\geqslant X_1(\epsilon).
$$
For $(t,x)\in \Omega_1 := \{ (t,x) : g^-(t)\leqslant x\leqslant h(t) -H_r -H_l -X_1(\epsilon),\ t>T''\}$, by
\eqref{hbd2} and \eqref{hbd} we have
$$
h^-(t) -x \geqslant R(t) -H_l -x \geqslant h(t) - H_r -H_l  -x \geqslant X_1(\epsilon),
$$
and so,
$$
v(t,x) \geqslant v^- (t,x) \geqslant (1-K_1 e^{-\delta t} ) V(t , X_1(\epsilon)) \geqslant
(1-K_1 e^{-\delta t} ) \Big( 1- \frac{\epsilon}{2P^0} \Big) \mbox{ for } (t,x)\in \Omega_1.
$$
Moreover, if we choose $T_2(\epsilon) >T''$ such that $2 P^0 K_1 e^{-\delta T_2(\epsilon)} <\epsilon$, then
\begin{equation}\label{v>1-epsilonP}
v(t,x)\geqslant \Big( 1- \frac{\epsilon}{2P^0}\Big)^2 > 1- \frac{\epsilon}{P^0} \mbox{ for }
(t,x)\in \Omega_1 \mbox{ and } t>T_2(\epsilon).
\end{equation}

\smallskip

{\bf Step 3}. To complete the proof of \eqref{u to P near 0}.
Denote $T_\epsilon := \max\{T_1(\epsilon), T_2(\epsilon)\}$ and $X_\epsilon := H_r + H_l +X_1(\epsilon)$, then
by \eqref{v<1+epsilon/P} and \eqref{v>1-epsilonP} we have
$$
|v(t,x)-1| \leqslant \frac{\epsilon}{P^0} \mbox{ for } c_l t \leqslant x\leqslant h(t) -X_\epsilon,\ t>T_\epsilon.
$$
This yields the estimate in \eqref{u to P near 0}.
\end{proof}

\begin{remark}\rm\label{rem:est h all beta}
We remark that the estimate \eqref{hbd} in Step 1 remains true even for large advection problems
(that is, even if $\bar{\beta}\geqslant B(\tilde{\beta})$). In fact, in the proof in Step 1 we
only use $V$ or $U$ (which exists for all $\beta\in \PP$ satisfying $\bar{\beta}\geqslant 0$ by Proposition
\ref{prop:tsw right}) to construct upper solutions. The estimate \eqref{hbd2}, however, is not true
for large advection case. The proof in Step 2 is also invalid since $c_l = \bar{\beta} -\bar{c} + \epsilon_0
> c_r = \bar{r} -\epsilon _0$ when $\bar{\beta}> B(\tilde{\beta})$ and $\epsilon_0 >0$.
\end{remark}

Using a similar argument as above we can obtain the following result.
\begin{prop}\label{pro:sigma12}
Assume that $\HH$, $\HHH$ and $\beta\in\PP$ satisfies $0\leqslant\bar{\beta}< \bar{c}$.
Assume further that spreading happens for the solution $(u,g,h)$. Then
\begin{itemize}
\item[(i)] there exists $C_1 >0$ such that
\begin{equation}\label{gh1}
|g(t) + L(t) |\leqslant C_1 \ \ \mbox{for all } t\geqslant0 ;
\end{equation}

\item[(ii)] for any small $\epsilon >0$, there exists $X'_\epsilon >0$ and $T'_\epsilon >0$ such that
\begin{equation}\label{u to P near 0 left}
\|u(t,\cdot ) - P(t) \|_{L^\infty ([g(t) + X'_\epsilon, 0])} \leqslant \epsilon \ \mbox{ when } t> T'_\epsilon.
\end{equation}
\end{itemize}
\end{prop}

\subsection{Asymptotic profiles of the (virtual) spreading solutions}\label{sub52}

\begin{thm}\label{thm:WHG}
Assume that $\HH$, $\HHH$ and $\beta\in\PP$ satisfies $0\leqslant\bar{\beta}
<B(\tilde{\beta})$. Assume further that spreading or virtual spreading happens.
Then there exists $H_1\in\R$  such that
\begin{equation}\label{HWt1}
\lim_{t\to\infty}[h(t) - R(t) ]= H_1 \ \ \lim_{t\to\infty} [h'(t) -r(t;\beta)] =0,
\end{equation}
\begin{equation}\label{WHt1}
\lim\limits_{t\to\infty} \| u(t,\cdot)- U(t,R(t)+H_1 -\cdot;\beta -r)\| _{L^\infty ( [c_l t, h(t)])}=0,
\end{equation}
where $R(t) = \int_0^t r(s;\beta) ds$ and $c_l$ is defined by \eqref{def:cl}.
Here we extend $U(t,z;\beta -r)$ to be zero for $z<0$.
\end{thm}

\begin{proof}
We use moving coordinate frames in our approach.

{\bf Step 1}. Using the moving coordinate $y :=x-R(t)$ we prove \eqref{HWt1}.
Set
$$
\begin{array}{l}
h_1(t):= h(t)-R(t), \quad g_1(t) :=g(t)-R(t) \mbox{ for } t\geqslant 0,\\
\mbox{and }  u_1(t,y):= u (t, y+R(t)) \mbox{ for } y\in[g_1(t),h_1(t)],\ t\geqslant 0.
\end{array}
$$
Then $({u_1},{g_1},{h_1})$ solves
\begin{equation}\label{pWH}
\left\{
\begin{array}{ll}
 {u_1}_t ={u_1}_{yy}+[r(t)-\beta(t)]{u_1}_y+ f(t,{u_1}), &  {g_1}(t)<y<{h_1}(t),\ t>0,\\
 {u_1}(t, y)= 0,\ {g'_1}(t)=-\mu(t){u_1}_y(t,y)-r(t), &  y={g_1}(t),\ t>0,\\
 {u_1}(t, y)= 0,\ {h'_1}(t)=-\mu(t){u_1}_y(t,y)-r(t), &  y={h_1}(t),\ t>0.
\end{array}
\right.
\end{equation}

For any $y_0\in\R$, the function $V_1 (t,y):=U(t,y_0- y; \beta -r)$ satisfies
\[
\left\{
\begin{array}{l}
 V_{1t} =V_{1yy}+[r(t)-\beta(t)]V_{1y}+ f(t,V_1), \quad   -\infty<y<y_0,\ t>0,\\
 V_1 (t, -\infty)= P(t),\ V_1 (t,y_0)=0,\ r(t)=-\mu(t)V_{1y} (t,y_0), \quad  t>0.
\end{array}
\right.
\]
We now consider the number of zeros of $\eta_1(t,y) := {u_1}(t,y) - V_1(t,y)$ in the interval
$J(t):=[{g_1}(t), \min\{y_0, {h_1}(t)\}]$.
Since ${g_1}(t)\to -\infty$ we have $V_1 (t,{g_1}(t))-P(t) \to0$ as $t\to\infty$.
Hence $ V_1 (t,{g_1}(t))>0= {u_1}(t,{g_1}(t))$ for all large $t$. The right end ${h_1}(t)$
of ${u_1}(t,y)$ may get across $y_0$ many times. By the zero number argument
(cf. \cite[Lemma 2.4]{DuLouZ}, \cite[Lemma 3.10 (a)]{GLZ}) we know that
$\mathcal{Z}_{J(t)}[\eta_1(t, \cdot) ]$ (which denotes the number of zeros of
the function $\eta_1 (t, \cdot)$ in $J(t)$) is finite, and it decreases strictly when ${h_1}(t)$
get across $y_0$. So ${h_1}(t)-y_0$ changes sign at most finitely
many times, namely, ${h_1}(t) >y_0$, or ${h_1}(t) <y_0$, or ${h_1}(t)\equiv y_0$ for all large $t$.
Since ${h_1}(t)$ is bounded by \eqref{hh1} and $y_0$ is an arbitrary point, we conclude that
${h_1}(t)$ converges as $t\to \infty$ to a number $H_1\in \R$. This proves the first limit in \eqref{HWt1}.

By the parabolic estimate as in \cite{DuLou, Wang} etc. we know that, for any $\tau>0 $,
$\|{h'}(t)\|_{C^{\nu/2}([\tau , \tau+1])}$ is bounded from above by a constant $C_1$ independent of $\tau$. Since
$r\in \PP_+\subset C^{\nu/2}([0,T])$ we conclude that there is a constant  $C>0$ independent of $\tau$ such that
\[
\|{h'_1}(t)\|_{C^{\nu/2}([\tau, \tau + 1])}\leqslant C.
\]
Combining with the convergence of ${h_1}(t)$ we obtain $h'_1(t)\to 0$ as $t\to \infty$. The second limit in \eqref{HWt1}
then follows.

\medskip

{\bf Step 2}. We use another moving coordinate $z:= x-h(t)$ to prove \eqref{WHt1}. Set
$$
g_2(t) := g(t)-h(t) \mbox{ for } t\geqslant 0,\ \ \mbox{and}\ \ u_2(t,z) := u(t, z+h(t)) \mbox{ for } z\in [g_2 (t), 0],\ t\geqslant 0.
$$
Then the pair $(u_2, g_2)$ solves
\begin{equation}\label{p u2}
\left\{
\begin{array}{ll}
 {u_2}_t ={u_2}_{zz}+[h'(t)-\beta(t)]{u_2}_z+ f(t,{u_2}), &  {g_2}(t)<z<0,\ t>0,\\
 {u_2}(t, z)= 0,\ {g'_2}(t)=- \mu(t){u_2}_z (t,z)-h'(t), &  z={g_2}(t),\ t>0,\\
 {u_2}(t, 0)= 0,\ h'(t) = -\mu(t){u_2}_z(t,0), &  t>0.
 \end{array}
\right.
\end{equation}
We will compare the $\omega$-limit functions of $u_2$ with
the function $V_2 (t,z):=U(t,-z;\beta -r)$, where $V_2$ solves
\[
\left\{
\begin{array}{l}
 V_{2t} =V_{2zz}+[r(t;\beta)-\beta(t)]V_{2z}+ f(t,V_2), \quad   -\infty<z<0,\ t\in \R,\\
 V_2 (t, -\infty)= P(t),\ V_2 (t,0)=0,\ r(t)=-\mu(t)V_{2z} (t,0), \quad  t\in \R.
\end{array}
\right.
\]

For any sequence of integers $\{m_n\}$ satisfying $m_n  \to \infty\ (n\to \infty)$,
since $u_2(m_n T +t, z)$ is bounded in $L^\infty$ norm, it follows from the $L^p$ theory, the Sobolev embedding theorem as well as
the Schauder estimates that, for any $K>0$, $\|u_2(m_n T +t,z)\|_{C^{1+\nu/2, 2+\nu}
([-K, K]\times [-K,0])}$ is bounded by a constant $C$ depending on $K$ but not on $n$.
Hence it has a subsequence converging in the space $C^{1, 2} ([-K, K]\times [-K,0])$.
Using Cantor's diagonal argument, there exist a function $w(t,z)\in C^{1+\nu/2, 2+\nu}(\R \times (-\infty, 0])$
and a subsequence of $\{m_n\}$, denoted again by $\{m_n\}$, such that
$u_2 (m_n T +t, z) \to w(t,z)$ in the topology of $C^{1,2}_{loc} (\R \times (-\infty,0])$.
Replacing $t$ by $m_n T +t$ in \eqref{p u2} and taking limit as $n\to\infty$  we obtain
\[
\left\{
\begin{array}{l}
 w_t = w_{zz} + [r(t)-\beta(t)]w_z+ f(t,w), \quad    -\infty <z<0,\ t\in \R,\\
 w(t,0)=0,\ r(t)=-\mu(t) w_z (t,0), \quad t\in \R.
\end{array}
\right.
\]
Consider the function $\eta_2 (t,z):= w(t,z) - V_2(t,z)$. It is clear that $z=0$
is a degenerate zero of $\eta_2 (t,\cdot)$ for all $t\in \R$. Hence, the zero number argument
(cf. \cite[Lemma 2.4]{DuLouZ}, \cite[Lemma 3.10 (a)]{GLZ}) indicates that $w(t,z)
\equiv V_2(t,z)$.  Since $\{m_n\}$ is an arbitrarily chosen consequence we have, for any $K>0$,
$$
\|u_2(t+nT, z) -V_2(t,z)\|_{L^\infty ([-K,K]\times [-K,0])} \to 0 \quad \mbox{as } n\to \infty,
$$
or, equivalently,
$$
\|u(t+nT, x) - U(t, h(t+nT)-x;\beta -r)\|_{L^\infty ([-K,K]\times [h(t+nT)-K,h(t+nT)])} \to 0 \quad \mbox{as } n\to \infty.
$$
Since $U(t,z;\beta -r)$ is $T$-periodic in $t$ we have
$$
\|u(t, \cdot) - U(t, h(t)-\cdot;\beta -r)\|_{L^\infty ([h(t)-K,h(t)])} \to 0 \quad \mbox{as } t\to \infty.
$$
Using the limit $h(t)-R(t)\to H_1$ in \eqref{HWt1} we obtain
\begin{equation}\label{u to U near h(t)}
\|u(t, \cdot) - U(t, R(t)+H_1 -\cdot;\beta -r)\|_{L^\infty ([h(t)-K,h(t)])} \to 0 \quad \mbox{as } t\to \infty.
\end{equation}
Here we have extended $U(t,z;\beta -r)$ to be zero for $z<0$.

\smallskip

{\bf Step 3}. Finally we prove \eqref{WHt1}.

For any given small $\epsilon >0$, by \eqref{u to P near 0} in Proposition \ref{pro:sigma01}, there exist
$X_\epsilon , T_\epsilon >0$ such that
$$
|u(t,x) - P(t)| \leqslant \epsilon\quad \mbox{for } c_l t\leqslant x\leqslant h(t) - X_\epsilon,\ t>T_\epsilon.
$$
Since $U(t,\infty; \beta -r)=P(t)$, there exists $X^*_\epsilon > X_\epsilon$ such that
$$
|U(t,R(t) +H_1 -x;\beta -r) -P(t)| \leqslant \epsilon\quad  \mbox{for } x\leqslant R(t)+ 2H_1 -X^*_\epsilon, \ t\in [0,T].
$$
Taking $T^*_\epsilon >T_\epsilon$ large such that $h(t) <R(t)+2H_1$ for $t>T^*_\epsilon$, then by combining the
above two inequalities  we obtain
$$
|u(t,x) - U(t,R(t) +H_1 -x;\beta -r) | \leqslant 2\epsilon\quad \mbox{for } c_l t\leqslant x\leqslant h(t)-X^*_\epsilon, \ t> T^*_\epsilon.
$$
Taking $K= X^*_\epsilon$ in \eqref{u to U near h(t)} we see that for some $T^{**}_\epsilon >T^*_\epsilon$, we have
$$
|u(t, x) - U(t, R(t)+H_1 -x;\beta -r)| \leqslant \epsilon \quad \mbox{for } h(t) -X^*_\epsilon \leqslant x \leqslant h(t),\ t>T^{**}_\epsilon.
$$
This prove \eqref{WHt1}.
\end{proof}

Using a similar argument as above one can obtain the following result.

\begin{thm}\label{thm:WGH}
Assume that $\HH$, $\HHH$ and $\beta\in\PP$ satisfies $0\leqslant\bar{\beta}
<\bar{c}$. Assume further that spreading happens. Then there exist $G_1\in\R$
such that \eqref{left spreading speed} and \eqref{profile convergence 1-left} hold.
\end{thm}

\smallskip

\noindent
{\it Proof of Theorem \ref{thm:profile of spreading sol}}. The conclusions in
Theorem \ref{thm:profile of spreading sol} follow from Theorems \ref{thm:WHG} and \ref{thm:WGH}.
\hfill $\square$

\end{document}